\documentclass[10pt]{amsart}

\usepackage{amsmath,amsfonts,latexsym,amssymb,mathrsfs,setspace,enumerate,color,cite,graphicx,epsf,qtree,amscd,fullpage}

\newtheorem{thm}{Theorem}
\newtheorem{lem}[thm]{Lemma}
\newtheorem{prop}[thm]{Proposition}
\newtheorem{cor}[thm]{Corollary}
\numberwithin{equation}{section}
\numberwithin{thm}{section}

\theoremstyle{definition}
\newtheorem{ex}[thm]{Example}
\newtheorem{conj}[thm]{Conjecture}

\newcommand{\rat}{\mathbb Q}
\newcommand{\real}{\mathbb R}

\newcommand{\alg}{\overline\rat}
\newcommand{\algt}{\alg^{\times}}

\newcommand{\intg}{\mathbb Z}
\newcommand{\nat}{\mathbb N}

\newcommand{\G}{\mathcal G}

\newcommand{\tors}{\mathrm{tors}}

\newcommand{\aaa}{{\bf A}}
\newcommand{\bbb}{{\bf B}}

\newcommand{\ttt}{\mathcal T}
\newcommand{\ppp}{\mathcal P}
\newcommand{\ooo}{\mathcal O}

\newcommand{\comment}[1]{}

\title[Optimal factorizations]{Optimal factorizations of rational numbers using factorization trees}

\author{Charles L. Samuels \and Tanner J. Strunk}

\address{Oklahoma City University, Department of Mathematics, 2501 N. Blackwelder, Oklahoma City, OK 73106, USA}
\email{clsamuels@okcu.edu, tjstrunk@my.okcu.edu}

\begin{document}

\begin{abstract}
	Let $m_t(\alpha)$ denote the $t$-metric Mahler measure of the algebraic number $\alpha$.
	Recent work of the first author established that the infimum in $m_t(\alpha)$ is attained by a single point $\bar\alpha = (\alpha_1,\ldots,\alpha_N)\in \alg^N$ 
	for all sufficiently large $t$.  Nevertheless, no efficient method for locating $\bar \alpha$ is known.  In this article, we define a new tree data structure, called a factorization tree, 
	which enables us to find $\bar\alpha$ when $\alpha\in \rat$.  We establish several basic properties of factorization trees, 
	and use these properties to locate $\bar\alpha$ in previously unknown cases.
\end{abstract}

\maketitle

\section{Introduction}

Suppose that $K$ is a number field and $v$ is a place of $K$ dividing the place $p$ of $\rat$.  Let $K_v$ and $\rat_p$ be their respective completions
so that $K_v$ is a finite extension of $\rat_p$.  We note the well-known fact that
\begin{equation*}
	\sum_{v\mid p} [K_v:\rat_p] = [K:\rat],
\end{equation*}
where the sum is taken over all places $v$ of $K$ dividing $p$.
Given $x\in K_v$, we define $\|x\|_v$ to be the unique extension of the $p$-adic absolute value on $\rat_p$ and set
\begin{equation} \label{NormalAbs}
	|x|_v = \|x\|_v ^{[K_v:\rat_p]/[K:\rat]}.
\end{equation}
If $\alpha\in K$, then $\alpha\in K_v$ for every place $v$, so we may define the {\it (logarithmic) Weil height} by
\begin{equation*}
	h(\alpha) = \sum_v \log^+ |\alpha|_v.
\end{equation*}
Due to our normalization of absolute values \eqref{NormalAbs}, this definition is independent of $K$, meaning that $h$ is well-defined
as a function on the multiplicative group $\algt$ of non-zero algebraic numbers.

It is well-known that $h(\alpha) = 0$ if and only if $\alpha$ is a root of unity, and it can easily be verified that $h(\alpha^n) = |n|\cdot h(\alpha)$ for all
integers $n$.  In particular, we see that $h(\alpha) = h(\alpha^{-1})$.
A theorem of Northcott \cite{Northcott} asserts that, given a positive real number $D$, there are only finitely many algebraic numbers $\alpha$ with $\deg\alpha\leq D$ and 
$h(\alpha)\leq D$.

The Weil height is closely connected to a famous 1933 problem of D.H. Lehmer \cite{Lehmer}.  The {\it (logarithmic) Mahler measure} of a non-zero algebraic number $\alpha$ is
defined by 
\begin{equation} \label{MahlerMeasure}
	m(\alpha) = [\rat(\alpha):\rat] \cdot h(\alpha).  
\end{equation}	
In attempting to construct large prime numbers, Lehmer came across the problem of determining whether there exists a sequence of
algebraic numbers $\{\alpha_n\}$, not roots of unity, such that $m(\alpha_n)$ tends to $0$ as $n\to\infty$.  This problem remains unresolved, although substantial evidence
suggests that no such sequence exists (see \cite{BDM, MossWeb, Schinzel, Smyth}, for instance).  This assertion is typically called Lehmer's conjecture.

\begin{conj}[Lehmer's Conjecture]
	There exists $c>0$ such that $m(\alpha) \geq c$ whenever $\alpha\in \algt$ is not a root of unity. 
\end{conj}

Dobrowolski \cite{Dobrowolski} provided the best known lower bound on $m(\alpha)$ in terms of $\deg\alpha$, while Voutier \cite{Voutier}
later gave a  version of this result with an effective constant.
Nevertheless, only little progress has been made on Lehmer's conjecture for an arbitrary algebraic number $\alpha$.

Dubickas and Smyth \cite{DubSmyth, DubSmyth2} were the first to study a modified version of the Mahler measure that has the triangle inequality.
They defined the {\it metric Mahler measure} by
\begin{equation*}
	m_1(\alpha) = \inf\left\{ \sum_{n=1}^N m(\alpha_n): N\in\nat,\ \alpha_n\in \algt\ \alpha = \prod_{n=1}^N\alpha_n\right\},
\end{equation*}
so that the infimum is taken over all ways of writing $\alpha$ as a product of algebraic numbers.  It is easily verified that $m_1(\alpha\beta) \leq m_1(\alpha) + m_1(\beta)$,
and that $m_1$ is well-defined on $\algt/\algt_{\tors}$.  It is further noted in \cite{DubSmyth2} that $m_1(\alpha) = 0$ if and only if $\alpha$ is a torsion point of $\algt$ and that
$m_1(\alpha) = m_1(\alpha^{-1})$ for all $\alpha\in\algt$.  These facts ensure that $(\alpha,\beta) \mapsto m_1(\alpha\beta^{-1})$ defines a metric on $\algt/\algt_\tors$.  
This metric induces the discrete topology if and only if Lehmer's conjecture is true.

The first author \cite{SamuelsCollection, SamuelsParametrized, SamuelsMetric} further extended this definition leading to the {\it $t$-metric Mahler measure}
\begin{equation} \label{TMetricMahler}
	m_t(\alpha) = \inf\left\{ \left( \sum_{n=1}^N m(\alpha_n)^t\right)^{1/t}:N\in\nat,\ \alpha_n\in \algt\ \alpha = \prod_{n=1}^N\alpha_n\right\}.
\end{equation}
In this context, we examined the function $t\mapsto m_t(\alpha)$ for a fixed algebraic number $\alpha$.  For instance, we showed that this function is everywhere continuous 
and infinitely differentiable at all but finitely many points.  If $G$ is a multiplicatively written Abelian group, we write
\begin{equation*}
	G^\infty = \left\{ (\alpha_1,\alpha_2,\ldots): \alpha_n\in  G,\ \alpha_n=1\mbox{ for all but finitely many } n\right\}.
\end{equation*}
An element of $\aaa = (\alpha_1,\alpha_2,\ldots) \in (\algt)^\infty$ is called a {\it representation of $\alpha$} if 
$\alpha = \prod_{n=1}^\infty \alpha_n$ and $m(\alpha_n) \geq m(\alpha_{n+1})$ for all $n\geq 1$.  
For simplicity, we shall often refer to the finite $N$-tuple $(\alpha_1,\dots,\alpha_N)$ as a representation of $\alpha$ if 
$(\alpha_1,\ldots,\alpha_N,1,1,\ldots)$ is a representation of $\alpha$.  In this case, we may simply write the product $\alpha_1\cdots\alpha_N$
to denote the corresponding representation.  If $\aaa = (\alpha_1,\alpha_2,\ldots)$ is a representation of $\alpha$ satisfying
\begin{equation*}
	m_t(\alpha) = \left( \sum_{n=1}^\infty m\left(\alpha_n\right)^t\right)^{1/t}
\end{equation*}
for all sufficiently large $t$, then we say that $\aaa$ is an {\it optimal representation}.
The following theorem is a consequence of \cite[Theorem 2.2]{SamuelsMetric}.

\begin{thm} \label{FiniteExceptional}
	If $\alpha$ is a non-zero algebraic number then $\alpha$ has an optimal representation.
\end{thm}

Optimal representations are important because they encode information about the arithmetic properties of $\alpha$.  
For instance, if $\alpha$ is a positive integer with prime factorization given by 
$\alpha = p_1p_2\cdots p_N$, the work of \cite{JankSamuels} asserts that $(p_1,p_2,\ldots,p_N)$ is an optimal representation of $\alpha$. 
Although specific optimal representations are known in some other special cases, current knowledge is extremely limited.   
Even the case where $\alpha\in \rat\setminus\intg$ is not well understood.
The proof of Theorem \ref{FiniteExceptional} does provide a method to search a list of candidates for an optimal
representation, but it alone gives little information on how to determine a sufficiently small list of candidates.
 On the other hand, a result of Jankauskas and the first author \cite{JankSamuels} provides a crucial improvement to Theorem \ref{FiniteExceptional} when $\alpha\in \rat$.

Suppose $\alpha = a/b$ is such that $a,b > 0$ and $\gcd(a,b) = 1$.  These assumptions on $\alpha$ may be made without loss of generality
and will be made throughout this article.  In this situation, we know that $m(\alpha) = \log \max\{a,b\}$.  An element
\begin{equation} \label{BasicFact}
	\aaa = \left( \frac{a_1}{b_1},\frac{a_2}{b_2},\ldots\right) \in (\rat^\times)^\infty
\end{equation}
is called a {\it factorization of $\alpha$} if the following four conditions hold.
\begin{enumerate}[(i)]
	\item $a_n,b_n \in \nat$ for all $n\in \nat$.
	\item $\alpha = \prod_{n=1}^\infty \frac{a_n}{b_n}$
	\item\label{Decreasing} $\max\{a_n,b_n\} \geq \max\{a_{n+1},b_{n+1}\}$ for all $n\in \nat$
	\item $\gcd(a_m,b_n) = 1$ for all $m,n\in\nat$.
\end{enumerate}
Of course, every factorization of $\alpha$ is a representation of $\alpha$.
By combining \cite[Theorem 1.2]{JankSamuels} and \cite[Theorem 2.2]{SamuelsParametrized},
we obtain an improvement to Theorem \ref{FiniteExceptional} when $\alpha\in \rat$.

\begin{thm}\label{Rationals2}
	If $\alpha$ is a positive rational number then $\alpha$ has an optimal factorization.
\end{thm}

As we remarked following the statement of Theorem \ref{FiniteExceptional}, the work of \cite{SamuelsMetric} provides a method to 
search a list of candidates for an optimal factorization.
Specifically, if $\mathcal A$ is known to contain at least one optimal factorization of $\alpha$, then we define the sequence of sets $\{\mathcal A_n\}$ as follows.
\begin{enumerate}[(I)]
	\item Let $\mathcal A_0 = \mathcal A$
	\item If $\mathcal A_n$ is given let $\mu_{n+1} = \min\{m(\alpha_{n+1}): (\alpha_1,\alpha_2,\ldots)\in \mathcal A_n\}$ and define
		\begin{equation*}
			\mathcal A_{n+1} = \left\{ (\alpha_1,\alpha_2,\ldots)\in\mathcal A_n: m(\alpha_{n+1}) = \mu_{n+1}\right\}.
		\end{equation*}
\end{enumerate}
It is shown in \cite{SamuelsMetric} that there exists $n$ such that  $\mathcal A_n$ contains only optimal factorizations of $\alpha$.

In view of Theorem \ref{Rationals2}, we could take $\mathcal A$ to be the set of all factorizations of $\alpha$ and we are guaranteed to eventually locate at least one 
optimal factorization.  But then $\#\mathcal A \gg \exp(N)$, where $N$ is the number of prime factors in $\alpha$, so the above method is highly time 
consuming when $\alpha$ has many prime factors.  

Our goal is to provide a slimmer list of factorizations to use as $\mathcal A$ in the above algorithm.
Although our method does not improve upon $\exp(N)$ in all cases, it does so in many special cases, whereas the above method always has $\#\mathcal A \gg \exp(N)$.
In the process, we will uncover certain tree data structures, called {\it factorization trees}, which we believe are of independent interest.

In section \ref{NewResults} we state the primary results of this article.  As understanding these statements requires some preparation, we use subsections \ref{DDS} and
\ref{FT} to formally introduce the notion of factorization tree mentioned above and to provide the appropriate graph theory background.  

There are two types of factorization trees that are particularly relevant to our study -- the {\it primitive factorization tree} and the {\it optimal factorization tree}.  These
are introduced in subsections \ref{PFT} and \ref{OFT}, respectively.  In each case, we study examples called the {\it maximal primitive factorization tree for $\alpha$} 
(denoted $\ppp_\alpha$) and the {\it canonical optimal factorization tree for $\alpha$} (denoted $\ooo_\alpha$) which are the subject of two of our main results, 
Theorems \ref{PrimitiveComplete} and \ref{OptimalOptimal}.  The former asserts several extremal properties of $\ppp_\alpha$ while the latter shows that both 
$\ppp_\alpha$ and $\ooo_\alpha$ may be used to locate optimal factorizations.

Each factorization tree gives rise to a natural quotient graph called a {\it measure class graph} which we discuss in subsection \ref{MCG}.  We state our final
main result, Theorem \ref{BinTree}, in that subsection as well.  

The proofs of all results are presented in section \ref{Proofs}.
Finally, we shall provide examples where our methods may be used to compute new optimal factorizations.

\section{New Results} \label{NewResults}

\subsection{Digraph Data Structures}\label{DDS}

Suppose that $G$ is a digraph with vertices $V(G)$ and edges $E(G)$.
Given any set $X$ and map $\nu: V(G)\to X$, the ordered pair $(G,\nu)$ is called a {\it digraph data structure for $X$}.
If $\G = (G,\nu)$ is such an object then $G$ is called the {\it skeleton of $\G$} and $\nu$ is called the {\it content map for $\G$}.
We shall sometimes write $V(\G) = V(G)$ or $E(\G) = E(G)$ for simplicity, but we emphasize that these sets do not depend on $\nu$.

Suppose that $\G_1 = (G_1,\nu_1)$ and $\G_2= (G_2,\nu_2)$ are digraph data structures for sets $X_1$ and $X_2$, respectively.  
Also assume that $f:X_1\to X_2$ is any map.  A map $\sigma:V(G_1)\to V(G_2)$ is a called an {\it $f$-homomorphism from $\G_1$ to $\G_2$} if the following hold.
\begin{enumerate}[(i)]
	\item\label{EdgePreserving} If $(g,h)\in E(G_1)$ then $(\sigma(g),\sigma(h))\in E(G_2)$.
	\item\label{ContentPreserving} If $r\in V(G_1)$ then $f(\nu_1(r)) = \nu_2(\sigma(r))$.
\end{enumerate}
We note that condition \eqref{ContentPreserving} asserts that we have a commutative diagram

$$\begin{CD}
	V(G_1)   @>\nu_1>>  X_1\\
	@VV\sigma V @VVfV\\
	V(G_2)  @>\nu_2>>   X_2
\end{CD}$$
\smallskip

If $\sigma$ is an $f$-homomorphism, we say that $\sigma$ is {\it faithful} if, for every $g_2,h_2\in \sigma(V(G_1))$ having $(g_2,h_2)\in E(G_2)$, there exists $(g_1,h_1)\in E(G_1)$ such that
$(\sigma(g_1),\sigma(h_1)) = (g_2,h_2)$.  If $\sigma$ satisfies the stronger condition
\begin{equation*}
	 (g,h)\in E(G_1) \mbox{ if and only if } (\sigma(g),\sigma(h))\in E(G_2)
\end{equation*}
then we say that $\sigma$ is {\it edge-preserving}.  In general, faithful does not imply edge-preserving, however if $\sigma$ is injective, it is easily seen
that the two conditions are equivalent.

If $\sigma$ is an $f$-homomorphism which is both bijective and edge-preserving, then we say that $\sigma$ is an {\it $f$-isomorphism}.
From our earlier remarks, an $f$-homomorphism $\sigma$ is an $f$-isomorphism if and only if it is both bijective and faithful.
In the special case where $X_1=X_2$ and $f$ is the identity map, $\sigma$ is simply called a {\it homomorphism} or {\it isomorphism}, respectively.
In the latter case, we write $\G_1 \cong \G_2$ and note that $\cong$ is certainly an equivalence relation on the set of all digraph data structures for $X_1$.

If $T$ is a rooted tree then $T$ is an example of a digraph and the resulting digraph data structure $\ttt = (T,\nu)$ is called
a {\it tree data structure for $X$}.  In this case, let $V^*(T)$ denote the set of all non-root vertices of $T$.  The {\it parenting map for $T$ (or $\ttt$)}
is the map $\phi:V^*(T) \to V(T)$ such that $\phi(s)$ is the parent vertex of $s$.  Therefore, $(r,s)$ is an edge of $T$ if and only if $s\in V^*(T)$ and $\phi(s) = r$.   
A vertex of $T$ which has no children is called a {\it leaf vertex}.

Our discussion of $f$-homomorphisms simplifies somewhat when considering tree data structures

\begin{thm} \label{TreeHM}
	Assume that $\ttt_1 = (T_1,\nu_1)$ and $\ttt_2 = (T_2,\nu_2)$ are tree data structures for some sets $X_1$ and $X_2$, respectively.  Let $f:X_1\to X_2$ be any map
 	and suppose that  $\phi_1$ and $\phi_2$ are the parenting maps for $\ttt_1$ and $\ttt_2$, respectively.
	A map $\sigma:V(T_1)\to V(T_2)$ is an $f$-homomorphism from $\ttt_1$ to $\ttt_2$ if and only if the following conditions hold.
	\begin{enumerate}[(a)]
		\item\label{ParentPreserving} If $r\in V^*(T_1)$ then $\sigma(\phi_1(r)) = \phi_2(\sigma(r))$
		\item\label{ContentPreserving2} If $r\in V(T_1)$ then $f(\nu_1(r)) = \nu_2(\sigma(r))$.
	\end{enumerate}
	Moreover, if $\sigma$ is an injective $f$-homomorphism then $\sigma$ is edge-preserving. 
	In particular, $\sigma$ is an $f$-isomorphism if and only if $\sigma$ is a bijective $f$-homomorphism.
\end{thm}

\subsection{Factorization Trees} \label{FT}

Suppose $\alpha = a/b\in \rat$ is such that $a,b > 0$ and $\gcd(a,b) = 1$.
If $p$ is a prime dividing $a$ or $b$ then we shall say that $p$ {\it divides the numerator} or  {\it divides the denominator} of $\alpha$, respectively.
Now let $p_1,p_2,\ldots, p_N$ be the not necessarily distinct primes dividing either $a$ or $b$ and assume that
\begin{equation*} \label{AlphaPrimes}
	p_1 \geq p_2 \geq \cdots \geq p_N.
\end{equation*}
Also let
\begin{equation*}
	\gamma(i) = \begin{cases} 1 & \mathrm{if}\  p_i \mid a \\ -1 & \mathrm{if}\ p_i \mid b. \end{cases}
\end{equation*}
We define a finite sequence of rational numbers $\alpha_n$, for $1\leq n\leq N$, by
\begin{equation*} \label{Alphan}
	\alpha_n = \prod_{i=1}^n p_i^{\gamma(i)}.
\end{equation*}
It is clear from the definition that $\alpha_N = \alpha$, and for completeness, we also define $\alpha_0 = 1$.

Suppose that
\begin{equation*}
	\aaa =  \left( \frac{a_1}{b_1},\frac{a_2}{b_2},\ldots\right)\quad\mathrm{and}\quad \bbb = \left( \frac{c_1}{d_1},\frac{c_2}{d_2},\ldots\right)
\end{equation*}
are factorizations of $\alpha_n$ and $\alpha_{n+1}$, respectively.  Assuming that $p_{n+1}$ divides $c_k$, where $k\in \nat$,
we say that $\aaa$ is a {\it direct subfactorization} of $\bbb$ if 
\begin{equation*}
	d_i = b_i \quad\mathrm{and} \quad c_i = \begin{cases} a_i & \mathrm{if}\ i\ne k \\ a_ip_{n+1} & \mathrm{if}\ i = k. \end{cases}
\end{equation*}
for all $i\in \nat$.   We use an analogous definition when $p_{n+1}$ divides $d_k$.  In either case, we write $\aaa < \bbb$.  

\begin{ex}\label{30/7Basic}
Consider $\alpha = \frac{30}{7}$. Here, $\frac{5}{7}\cdot \frac{3}{1}$ is a direct subfactorization of $\frac{5}{7}\cdot \frac{3}{1}\cdot \frac{2}{1}$.
Indeed, the former is a factorization of $\alpha_3$ while the latter is a factorization of $\alpha_4$, and all numerators and denominators are equal except the 
one containing the fourth largest prime $2$.  For similar reasons, $\frac{5}{7}\cdot \frac{3}{1}$ is a direct subfactorization of $\frac{10}{7}\cdot \frac{3}{1}$. 

On the other hand, $\frac{3}{7}\cdot \frac{5}{1}$ is not a direct subfactorization of $\frac{5}{7}\cdot \frac{3}{1}\cdot \frac{2}{1}$ since these factorizations differ at two distinct entries.
Also, $\frac{2}{7}\cdot \frac{5}{1}$ is not a direct subfactorization of $\frac{2}{7}\cdot \frac{5}{1}\cdot \frac{3}{1}$ since the former is not a factorization of $\alpha_n$ for any $n$.
\end{ex}

If $\aaa$ is a direct subfactorization of $\bbb$, we emphasize the implicit assumption that $\aaa$ and $\bbb$ are factorizations of $\alpha_n$ and $\alpha_{n+1}$, respectively,
for some $0\leq n < N$.  In particular, this definition depends on $\alpha$ even though we have suppressed this dependency in our notation.  
This will be common practice throughout this article.
	
To generalize our definition of direct subfactorization, assume that 
\begin{equation*}
	\aaa =  \left( \frac{a_1}{b_1},\frac{a_2}{b_2},\ldots\right)\quad\mathrm{and}\quad \bbb = \left( \frac{c_1}{d_1},\frac{c_2}{d_2},\ldots\right)
\end{equation*}
are factorizations of $\alpha_n$ and $\alpha_{m}$, respectively, where $n < m$.  We say that $\aaa$ is a {\it subfactorization} of $\bbb$ if
there exist factorizations $\aaa_{n+1},\aaa_{n+2},\ldots,\aaa_{m-1}$ of $\alpha_{n+1},\alpha_{n+2},\ldots,\alpha_{m-1}$, respectively, such that 
\begin{equation} \label{GeneralSub}
	\aaa < \aaa_{n+1} < \aaa_{n+2} < \cdots < \aaa_{m-1} < \bbb.
\end{equation}
The following observation is clear from the definitions.

\begin{prop} \label{Subfact}
	Suppose $\alpha$ is a rational number, $\aaa$ is a factorization of $\alpha_m$, and $\bbb$ is a factorization of $\alpha_n$.   Then
	$\aaa$ is a direct subfactorization of $\bbb$ if and only if $\aaa$ is a subfactorization of $\bbb$ and $n = m+1$.
\end{prop}

In view of Proposition \ref{Subfact}, there is no ambiguity in writing $\aaa < \bbb$ whenever $\aaa$ is a subfactorization of $\bbb$.  It is straightforward to verify that 
$<$ defines a strict partial ordering on
\begin{equation*}
	\mathfrak F_\alpha = \left\{\aaa: \aaa\mbox{ is a factorization of } \alpha_n \mbox{ for some } 0\leq n\leq N \right\}
\end{equation*}
and that $\mathfrak F_\alpha$ has a unique minimal element, namely $(1,1,\ldots)$.   Moreover, an element $\aaa\in \mathfrak F_\alpha$ is  maximal if and only
if $\aaa$ is a factorization of $\alpha$.

Suppose that $\ttt = (T,\nu)$ is a tree data structure for $\mathfrak F_\alpha$ with parenting map $\phi$.  
$\ttt$ is called a {\it factorization tree for $\alpha$} if the following conditions hold.
\begin{enumerate}[(i)]
	\item\label{RootVertex} If $r$ is the root vertex of $T$ then $\nu(r) = (1,1,\ldots)$.
	\item\label{HasKids} If $n < N$ and $r$ is a vertex of $T$ such that $\nu(r)$ is a factorization of $\alpha_n$, then $r$ has at least one child.
	\item\label{UniqueKids} If $r$ and $s$ are vertices of $T$ such that $\nu(r) = \nu(s)$ and $\phi(r) = \phi(s)$ then $r = s$.
	\item\label{ParentRelations} If $r$ is a non-root vertex of $T$ then $\nu(\phi(r))$ is a direct subfactorization of $\nu(r)$.
\end{enumerate}

For clarification purposes, \eqref{UniqueKids} asserts that two vertices $r$ and $s$ containing the same factorization $\nu(r) = \nu(s)$ and having the same parent $\phi(r) = \phi(s)$ must
have $r=s$.  However, neither $\nu$ nor $\phi$ is an injection in general.  Certainly $\phi$ is an injection if and only if every vertex has at most one child.  Later (see Theorem \ref{Injection}),
we shall provide a sufficient condition for $\nu$ to be an injection.  

If $r$ is a vertex of $T$ such that $\nu(r)$ is a factorization of $\alpha$, then it follows from \eqref{ParentRelations} that $r$ has no children.  
Indeed, if $s$ is a child of $r$ then \eqref{ParentRelations} implies that $\nu(r) < \nu(s)$ contradicting the fact $\nu(r)$ is maximal with respect to the subfactorization relation.
Combining this with \eqref{HasKids}, we observe that the leaf vertices of $T$ are precisely those vertices $r$ such that $\nu(r)$ is a factorization of $\alpha$.
We also observe that $V(T)$ is a finite set.

It is not difficult to check that every rational number has a factorization tree.  We note the following example.

\begin{ex} \label{30/7}  The following is a factorization tree for $30/7$.

\bigskip
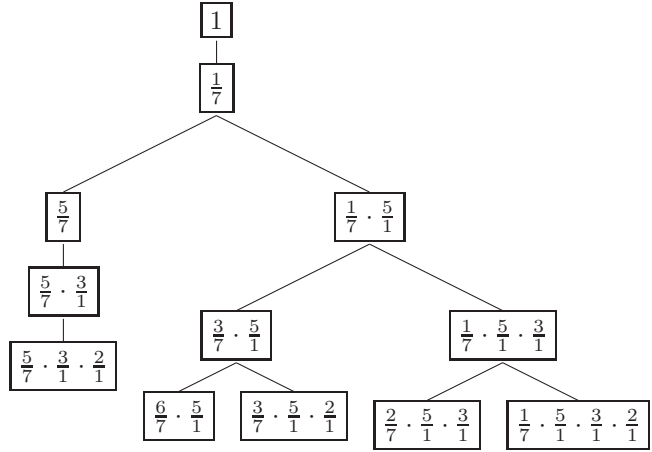
\Tree[.\fbox{$1$} 
		[.\fbox{$\frac{1}{7}$} 
			[.\fbox{$\frac{5}{7}$} 
				[.\fbox{$\frac{5}{7}\cdot \frac{3}{1}$}
					\fbox{$\frac{5}{7}\cdot\frac{3}{1}\cdot\frac{2}{1}$}
				]
			]
			[.\fbox{$\frac{1}{7}\cdot\frac{5}{1}$}
				[.\fbox{$\frac{3}{7}\cdot\frac{5}{1}$}
					\fbox{$\frac{6}{7}\cdot\frac{5}{1}$}
					\fbox{$\frac{3}{7}\cdot\frac{5}{1}\cdot \frac{2}{1}$}
				]
				[.\fbox{$\frac{1}{7}\cdot\frac{5}{1}\cdot\frac{3}{1}$}
					\fbox{$\frac{2}{7}\cdot\frac{5}{1}\cdot\frac{3}{1}$}
					\fbox{$\frac{1}{7}\cdot\frac{5}{1}\cdot\frac{3}{1}\cdot \frac{2}{1}$}
				]
			]
		]   
	]
	
\end{ex}

Our next theorem exhibits the strength of our definition of homomorphism when applied to factorization trees.

\begin{thm} \label{UniqueHMThm}
	Assume that $\ttt_1$ and $\ttt_2$ are factorization trees for $\alpha$.
	If $\sigma$ is a homomorphism from $\ttt_1$ to $\ttt_2$ then the following hold.
	\begin{enumerate}[(i)]
		\item\label{AllInjective} $\sigma$ is an injective edge-preserving homomorphism.
		\item\label{UniqueHM} If $\tau$ is another homomorphism from $\ttt_1$ to $\ttt_2$ then $\sigma = \tau$.
	\end{enumerate}
\end{thm}

While straightforward, it is useful to observe the following corollary.

\begin{cor} \label{UniqueHMCor}
	Assume that $\ttt_1$ and $\ttt_2$ are factorization trees for $\alpha$.
	\begin{enumerate}[(i)]
		\item A map $\sigma:V(T_1)\to V(T_2)$ is an isomorphism if and only if it is a surjective homomorphism.
		\item If $\ttt_1\cong \ttt_2$ then there exists a unique isomorphism from $\ttt_1$ to $\ttt_2$.
		\item The identity map is the only isomorphism from a factorization tree $\ttt$ to itself.
	\end{enumerate}
\end{cor}

\subsection{Primitive Factorization Trees} \label{PFT}

As part of our search for optimal factorizations of $\alpha$, there are two specific types of factorization trees that will be useful to study.  
We discuss the first of these trees in this section.

A factorization
\begin{equation} \label{AFact}
	\aaa =  \left( \frac{a_1}{b_1},\frac{a_2}{b_2},\ldots\right)
\end{equation}
is called {\it primitive} if $\max\{a_i,b_i\}$ is prime or equal to $1$ for all $i\in \nat$.  We define
\begin{equation*}
	\mathfrak P_\alpha = \left\{\aaa: \aaa\mbox{ is a primitive factorization of } \alpha_n \mbox{ for some } 0\leq n\leq N \right\}.
\end{equation*}
A factorization tree $\ttt = (T,\nu)$ for $\alpha$ is called a {\it primitive factorization tree} if 
\begin{equation} \label{PrimitiveDef}
	\nu(V(T)) \subseteq \mathfrak P_\alpha.
\end{equation} 
For instance, Example \eqref{30/7} given in the previous subsection is a primitive factorization tree.

Our goal for this subsection is to define a {\it maximal primitive factorization tree} $\ttt = (T,\nu)$ which we will show satisfies $\nu(V(T)) = \mathfrak P_\alpha$.
Assume that $n$ is a positive integer with $n<N$ and $\aaa$ is the factorization $\alpha_n$ given in \eqref{AFact}.
Further suppose that $p_{n+1}$ divides the numerator of $a$.  Define the collection of factorizations of $\alpha_{n+1}$ by
\begin{equation*}
	\delta(\aaa) = \left\{\left(\frac{a_1}{b_1},\cdots,\frac{a_{k-1}}{b_{k-1}},\frac{a_kp_{n+1}}{b_k},\frac{a_{k+1}}{b_{k+1}},\cdots\right): k\in \nat \mbox{ and } a_kp_{n+1} < b_k\right\}.
\end{equation*}
If $\aaa$ is a primitive factorization, then it is easily seen that all elements of $\delta(\aaa)$ are primitive factorizations, and moreover, $\aaa$ is a direct subfactorization
of every element in $\delta(\aaa)$.  Albeit trivial, it is also worth noting that $\delta(\aaa)$ is empty precisely when $a_kp_{n+1} \geq b_k$ for all $k\in \nat$.  

We may assume that
\begin{equation*}
	\aaa =  \left( \frac{a_1}{b_1},\frac{a_2}{b_2},\ldots, \frac{a_\ell}{b_\ell},1,1,\cdots\right),
\end{equation*}
where $a_\ell/b_\ell \ne 1$, and define $\epsilon(\aaa)$ to be the singleton set
\begin{equation*}
	\epsilon(\aaa) = \left\{\left( \frac{a_1}{b_1},\frac{a_2}{b_2},\ldots, \frac{a_\ell}{b_\ell},\frac{p_{n+1}}{1},1,1,\cdots\right)\right\}.
\end{equation*}
In the case where $p_{n+1}$ divides the denominator of $\alpha$, we define $\delta(\aaa)$ and $\epsilon(\aaa)$ in an analogous way.
Next set
\begin{equation*}
	\Delta(\aaa) = \delta(\aaa) \cup\epsilon(\aaa).
\end{equation*}
If $\aaa$ is primitive, we again have that all factorizations in $\Delta(\aaa)$ are primitive, and still, $\aaa$ is a direct subfactorization
of every element in $\Delta(\aaa)$.   Unlike $\delta(\aaa)$, we know that $\Delta(\aaa)$ is necessarily non-empty. 

We are careful to note that all three sets $\delta(\aaa$), $\epsilon(\aaa)$ and $\Delta(\aaa)$ require that $\aaa$ be a factorization of $\alpha_n$ for some $0\leq n < N$.
In particular, they depend on both $\alpha$ and $n$, although we have suppressed these dependencies in our notation.  In any examples seen in this article, 
we will apply these functions only to factorizations of the form $\nu(r)$, where $r$ is a vertex of a factorization tree for $\alpha$, so we will never encounter any ambiguity.

If $r$ is a vertex of a factorization tree $\ttt$, we shall write $\mathcal C(r)$ to denote the set of all children of $r$.
For a rational number $\alpha$, a {\it maximal primitive factorization tree for $\alpha$} is a factorization tree $\ttt$ for $\alpha$ 
such that $\nu(\mathcal C(r)) = \Delta(\nu(r))$ for all non-leaf vertices $r$ of $\ttt$.  Clearly every rational number has a maximal primitive factorization tree, 
and moreover, our next theorem shows that such trees are unique up to isomorphism.

\begin{thm} \label{CPFTIso}
	If $\alpha$ is a rational number and $\ttt_1$ and $\ttt_2$ are maximal primitive factorization trees for $\alpha$ then $\ttt_1 \cong \ttt_2$.
\end{thm}

In view of Theorem \ref{CPFTIso}, we shall now write $\ppp_\alpha$ to denote the maximal primitive factorization tree for $\alpha$.  
Strictly speaking, $\ppp_\alpha$ is an isomorphism class of factorization trees, but all of our results are independent of the choice of representative.
Hence, we shall often simply write $\ppp_\alpha$ to denote some particular maximal primitive factorization tree.

It is fairly clear from the definition that $\ppp_\alpha$ is a primitive factorization tree.  Our next result shows that $\ppp_\alpha$ has several maximality properties.   
We say that a rational number $\alpha = a/b$, with $\gcd(a,b) = 1$, is {\it square-free} if $a$ and $b$ are both square free.

\begin{thm}\label{PrimitiveComplete}
	If $\alpha$ is a rational number then $\ppp_\alpha$ is a primitive factorization tree.  Moreover, the following conditions hold.
	\begin{enumerate}[(i)]
		\item \label{PrimsAllOver} $\nu(V(\ppp_\alpha)) = \mathfrak P_\alpha$.
		\item \label{CPFTBiggest} If $\ttt$ is a primitive factorization tree for $\alpha$ then there exists a unique homomorphism $\sigma:V(\ttt)\to V(\ppp_\alpha)$.
			Moreover, $\sigma$ is injective and edge-preserving.
		\item \label{UniqueSqF} Suppose $\alpha$ is square-free and $\ttt = (T,\nu_0)$ is a primitive factorization tree for $\alpha$.   If every primitive
			factorization of $\alpha$ belongs to $\nu_0(V(T))$ then $\ttt\cong \ppp_\alpha$.
	\end{enumerate}
\end{thm}


It is important to note that the square-free assumption in \eqref{UniqueSqF} cannot be removed.  In fact, if $\alpha$ fails to be square-free, then the same factorization of $\alpha$ may appear
in two distinct vertices of $\ppp_\alpha$.  However, a new smaller tree can be formed by removing certain vertices with duplicate factorizations.  The resulting tree $\ttt = (T,\nu_0)$
still satisfies $\nu_0(V(T)) = \mathfrak P_\alpha$, and hence satisfies the assumption of \eqref{UniqueSqF}, but it is not isomorphic to $\ppp_\alpha$.

\subsection{Optimal Factorization Trees}\label{OFT}

Recall that a factorization
\begin{equation*}
	\aaa =  \left( \frac{a_1}{b_1},\frac{a_2}{b_2},\ldots\right)
\end{equation*}
of $\alpha$ is called {\it optimal} if there exists a positive real number $T$ such that
\begin{equation} \label{Optimality}
	m_t(\alpha) = \left(\sum_{n=1}^\infty m\left(\frac{a_n}{b_n}\right)^t\right)^{1/t}\mbox{ for all } t\geq T.
\end{equation}
As we noted in the introduction every rational number has an optimal factorization.  We also note the following important result.

\begin{thm} \label{EOFPrim}
	If $\alpha$ is a rational number then every optimal factorization of $\alpha$ is primitive.
\end{thm}

As discussed in the introduction, our agenda is to use factorization trees to search for optimal factorizations of the rational number $\alpha$.  
For this purpose, we let
\begin{equation*}
	\mathfrak O_\alpha = \left\{\aaa: \aaa\mbox{ is an optimal factorization of } \alpha_n \mbox{ for some } 0\leq n\leq N \right\}.
\end{equation*}
In view of Theorem \ref{Rationals2}, $\mathfrak O_\alpha$ contains at least $N+1$ elements and is, in particular, non-empty.
From Theorem \ref{EOFPrim} we note that
\begin{equation*}
	\mathfrak O_\alpha \subseteq \mathfrak P_\alpha \subseteq \mathfrak F_\alpha.
\end{equation*}
A factorization tree $\ttt$ for $\alpha$ is called an {\it optimal factorization tree} if 
\begin{equation} \label{OptimalDef}
	\mathfrak O_\alpha \subseteq \nu(V(\ttt)).
\end{equation}
We should not regard the definition of optimal as an analog of primitive.  Indeed, the set containment in the definition of primitive \eqref{PrimitiveDef}
points in the opposite direction from that of \eqref{OptimalDef}.  
Because of this discrepancy, there exist optimal factorization trees which are not primitive in spite of the fact that $\mathfrak O_\alpha \subseteq \mathfrak P_\alpha$.
Indeed, a factorization tree may satisfy \eqref{OptimalDef} but still may have vertices $r$ such that $\nu(r)$ is not primitive.

By applying Theorem \ref{PrimitiveComplete}\eqref{PrimsAllOver} and Theorem \ref{EOFPrim}, we are already familiar with one particular 
optimal factorization tree.

\begin{cor} \label{OptimalPrimitive}
	If $\alpha$ is a rational number then $\ppp_\alpha$ is both optimal and primitive.
\end{cor}

As our goal is to locate an optimal factorization for a given rational number $\alpha$, Corollary \ref{OptimalPrimitive} helps us considerably.  Indeed, 
if we can determine $\ppp_\alpha$ then we know that every optimal factorization of $\alpha$ lies among the leaf vertices of $\ppp_\alpha$.  Then we may use the techniques of
\cite{SamuelsMetric} to search the leaf vertices of $\ppp_\alpha$ for an optimal factorization.  

Recall that, to obtain the maximal primitive factorization tree $\ppp_\alpha$, we imposed the restriction $\nu(\mathcal C(r)) = \Delta(\nu(r))$ on an arbitrary factorization tree.
However, it is possible to impose a stronger restriction on $\nu(\mathcal C(r))$ while still preserving the optimality of the tree.  This further abbreviates our search for 
optimal factorizations.  A factorization tree $\ttt$ for $\alpha$ is called a {\it canonical optimal factorization tree} for $\alpha$ if
\begin{equation*}
	\nu(\mathcal C(r)) = \begin{cases} \delta(\nu(r)) & \mbox{if } \delta(\nu(r)) \ne \emptyset \\
								\epsilon(\nu(r)) & \mbox{if } \delta(\nu(r)) = \emptyset.
					\end{cases}
\end{equation*}
for all non-leaf vertices $r$ of $\ttt$.  As was the case with our maximal primitive factorization trees, canonical optimal factorization trees are unique up to isomporphism.

\begin{thm} \label{COFTIso}
 	If $\alpha$ is a rational number and $\ttt_1$ and $\ttt_2$ are canonical optimal factorization trees for $\alpha$ then $\ttt_1 \cong \ttt_2$.
\end{thm}

We now write $\ooo_\alpha$ to denote the canonical optimal factorization tree for $\alpha$.  As in the case of $\ppp_\alpha$, although $\ooo_\alpha$ is an 
isomorphism class of factorization trees, we use this notation to denote a specific canonical optimal factorization tree.  All of our results are indeed independent 
of the choice of representative.

Certainly $\ooo_\alpha$ is a primitive factorization tree, so Theorem \ref{PrimitiveComplete} \eqref{CPFTBiggest} asserts the existence of a unique injective edge-preserving
homomorphism $\sigma:V(\ooo_\alpha) \to V(\ppp_\alpha)$.  It is trivial to provide examples where this map fails to be a surjection, meaning that $\ooo_\alpha$ is, in general,
strictly smaller than $\ppp_\alpha$.  Therefore, the following result is a direct improvement over Corollary \ref{OptimalPrimitive}.

\begin{thm} \label{OptimalOptimal}
	If $\alpha$ is a rational number then $\ooo_\alpha$ is both optimal and primitive.
\end{thm}

\subsection{Measure Class Graphs} \label{MCG}

Suppose $\aaa = (a_1/b_1,a_2/b_2,\ldots)\in (\rat^\times)^\infty$.  We define the {\it measure} of $\aaa$ by
\begin{equation*}
	m(\aaa) = \left( m\left(\frac{a_1}{b_1}\right), m\left(\frac{a_2}{b_2}\right), \cdots\right)
\end{equation*}
and note that $m(\aaa) \in \real^\infty$.  In the previous two subsections, we provided examples of trees which are guaranteed to contain all optimal 
factorizations of a given rational number $\alpha$.  In some cases, however, we may only be interested in determining the measure of each optimal factorization,
so $\ppp_\alpha$ or $\ooo_\alpha$ contains more information than is required.  Hence, we are motivated to consider trees that contain only information about the measures of
factorizations, called {\it measure class trees}, to be defined momentarily.

If $\bbb = (c_1/d_1,c_2/d_2,\ldots)$ is another element of $\rat^\infty$, we say that $\aaa$ is {\it measure equivalent} to $\bbb$ if the following conditions hold.
\begin{enumerate}[(i)]
	\item $\aaa$ and $\bbb$ are both factorizations of $\alpha$ for some $\alpha\in \rat$.
	\item $m(\aaa) = m(\bbb)$.  
\end{enumerate}
In this case, we write $\aaa\sim \bbb$.  It is clear that $\sim$ is an equivalence relation on $\mathfrak F_\alpha$ and we write $[\aaa]$ to denote the 
equivalence class of $\aaa$.  Write
\begin{equation*}
	\overline{\mathfrak F}_\alpha = \left\{ [\aaa]: \aaa\in \mathfrak F_\alpha\right\}
\end{equation*}
and define the surjection $f:\mathfrak F_\alpha\to \overline{\mathfrak F}_\alpha$ by $f(\aaa) = [\aaa]$.

Suppose $\ttt = (T,\nu)$ is a factorization tree for $\alpha$ and $\G = (G,\mu)$ is a digraph data structure for $\overline{\mathfrak F}_\alpha$.
$\G$ is called a {\it measure class graph} for $\ttt$ if
\begin{enumerate}[(i)]
	\item \label{MuInjection} $\mu$ is an injection.
	\item\label{SurjFaithful} There exists a surjective faithful $f$-homomorphism $\pi:V(T) \to V(G)$.
\end{enumerate}
We remind the reader that condition \eqref{SurjFaithful} asserts the existence of the commutative diagram

$$\begin{CD}
	V(T)   @>\nu>>  \mathfrak F_\alpha\\
	@VV\pi V @VVfV\\
	V(G)  @>\mu>>   \overline{\mathfrak F}_\alpha
\end{CD}$$
\smallskip

In this case, $\pi$ is called the {\it projection map from $\ttt$ onto $\G$} and we note that this map is unique.  After all, if $\pi_1$ and $\pi_2$ are projection maps
from $\ttt$ onto $\G$, then
\begin{equation*}
	\mu(\pi_1(r)) = f(\nu(r)) = \mu(\pi_2(r)).
\end{equation*}
But $\mu$ is injection so we see that $\pi_1(r) = \pi_2(r)$.  Moreover, every factorization tree has a measure class graph.  
Indeed, we may define an equivalence relation on $V(\ttt)$ by declaring $r\sim s$ precisely when 
$\nu(r) \sim \nu(s)$.  Now write $[r]$ for the equivalence class containing $r$.  
The resulting quotient digraph having content map $[r] \mapsto [\nu(r)]$ is easily verified to be a measure class graph for $\ttt$.
The following theorem asserts that measure class graphs are unique up to isomorphism.

\begin{thm} \label{MCGUnique}
	Let $\ttt_1$ and $\ttt_2$ be factorization trees for $\alpha$ with $\ttt_1\cong\ttt_2$.  
	If $\mathcal G_1$ and $\mathcal G_2$ are measure class graphs for $\ttt_1$  and $\ttt_2$, respectively, then $\G_1\cong \G_2$.
\end{thm}

In view of Theorem \ref{MCGUnique}, we shall simply write $\overline \ttt$ for the measure class graph of $\ttt$.  Note that the the skeleton of $\overline \ttt$
is simply the quotient graph of $T$ by the equivalence relation on vertices given by $r\sim s \iff \nu(r) \sim \nu(s)$.

In general, $\overline \ttt$ is not a tree as it is possible for
a factorization to have two direct subfactorizations that are not equivalent.  For instance, taking $\alpha = 4/15$, we observe that
\begin{equation*}
	\aaa = \left(\frac{1}{5},\frac{2}{3},\frac{2}{1}\right)
\end{equation*}
is a factorization of $\alpha_4$.  However, $\aaa$ has both
\begin{equation*}
	\bbb_1 = \left(\frac{1}{5},\frac{2}{3}\right)\quad\mathrm{and}\quad \bbb_2 = \left(\frac{1}{5},\frac{1}{3},\frac{2}{1}\right)
\end{equation*}
as direct subfactorizations, which are not equivalent.  Therefore, if $h\in V(\overline \ttt)$ with $\nu(h) = [\aaa]$ then there can exist distinct vertices $g_1, g_2 \in V(\G)$,
one having $\mu(g_1) = [\bbb_1]$ and the other having $\mu(g_2) = [\bbb_2]$.  As $\pi$ must be faithful and surjective, both $(g_1,h)$ and $(g_2,h)$ are edges of $\G$.
However, if $\ttt$ is known to be primitive and $\alpha$ is square-free, then this situation cannot occur.  A tree $T$ is {\it binary} if every vertex has at most two children.

\begin{thm} \label{BinTree}
	If $\ttt$ is a primitive factorization tree for the square-free rational number $\alpha$ then the skeleton of $\overline \ttt$ is a binary tree.
\end{thm}

In the case of Theorem \ref{BinTree}, we say that $\overline\ttt$ is a the {\it measure class tree} for $\ttt$.
As $\ppp_\alpha$ and $\ooo_\alpha$ are both primitive factorization trees, Theorem \ref{BinTree} asserts that $\overline{\ppp_\alpha}$ and $\overline{\ooo_\alpha}$
are also binary trees.

\section{Proofs of Results} \label{Proofs}

\subsection{Digraph Data Structures} 

The proof of Theorem \ref{TreeHM} is extremely simple but we include it here for completeness purposes.

\begin{proof}[Proof of Theorem \ref{TreeHM}]
	The proof of the first assertion requires only showing that \eqref{ParentPreserving} $\iff$ \eqref{EdgePreserving} which is straightforward.
	Assuming now that $\sigma$ is an injective homomorphism, suppose that $r,s\in V(T_1)$ are such that $(\sigma(r),\sigma(s))\in E(T_2)$.
	Hence, property \eqref{ParentPreserving} gives
	\begin{equation*}
		\sigma(r) = \phi_2(\sigma(s)) = \sigma(\phi_1(s)).
	\end{equation*}
	Since $\sigma$ is assumed to be injective, we obtain $\phi_1(s) = r$ so that $(r,s)\in E(T_1)$ as required.
\end{proof}

\subsection{Factorization Trees}

We begin by noting that the proof of Proposition \ref{Subfact} follows directly from the definition of subfactorization and direct subfactorization.
Hence, we proceed to establish the important properties found in Theorem \ref{UniqueHMThm} regarding homomorphisms from one factorization tree to another.

\begin{proof}[Proof of Theorem \ref{UniqueHMThm}]
	Set $\ttt_1 = (T_1,\nu_1)$ and $\ttt_2 = (T_2,\nu_2)$.
	To prove \eqref{AllInjective}, assume that $\sigma(r) = \sigma(s)$ and let $\phi_1$ and $\phi_2$ be the parenting maps for $\ttt_1$ and $\ttt_2$, respectively.  
	We know that $\nu_2(\sigma(r)) = \nu_2(\sigma(s))$, and from the definition of homomorphism, we get $\nu_1(r) = \nu_1(s)$.
	We may assume that this is a factorization of $\alpha_n$.  By applying properties \eqref{RootVertex} and \eqref{ParentRelations},
	we find that $\phi_1^n(r)$ and $\phi_1^n(s)$ both equal the root vertex of $\ttt_1$.
	
	If $r\ne s$ then we may assume that $k$ is the smallest positive integer such that $\phi_1^k(r) = \phi_1^k(s)$.  Therefore, $\phi_1^{k-1}(r)$ and $\phi_1^{k-1}(s)$ 
	are both children of this vertex.  We also have that
	\begin{equation*}
		\nu_1(\phi_1^{k-1}(r)) = \nu_2(\sigma(\phi_1^{k-1}(r))) = \nu_2(\phi_2^{k-1}(\sigma(r))),
	\end{equation*}
	and the analogous equalities for $s$ yield $\nu_1(\phi_1^{k-1}(r)) = \nu_1(\phi_1^{k-1}(s))$.  Applying property \eqref{UniqueKids} in the definition of factorization tree,
	we obtain that $\phi_1^{k-1}(r) = \phi_1^{k-1}(s)$, a contradiction.  It follows from Theorem \ref{TreeHM} that $\sigma$ is edge-preserving.

	Now suppose that $\tau$ is a homomorphism from $\ttt_1$ to $\ttt_2$.  Assume that $r\in V(T_1)$ is such that $\nu_1(r)$ is a factorization of $\alpha_n$.
	Then we know that both $\nu_2(\sigma(r))$ and $\nu_2(\tau(r))$ are factorizations of $\alpha_n$ as well.
	Therefore, $\phi_2^n(\sigma(r))$ and $\phi_2^n(\tau(r))$ both equal the root vertex of $\ttt_2$. 
	
	Following our proof of \eqref{AllInjective}, if $\sigma(r)\ne \tau(r)$ then we may assume that $k$ is the smallest positive integer such that $\phi_2^k(\sigma(r)) = \phi_2^k(\tau(r))$.  
	Then $\phi_2^{k-1}(\sigma(r))$ and $\phi_2^{k-1}(\tau(r))$ are both children of this vertex.  We see that
	\begin{equation*}
		\nu_2(\phi_2^{k-1}(\sigma(r))) = \nu_2(\sigma(\phi_1^{k-1}(r))) = \nu_1(\phi_1^{k-1}(r)),
	\end{equation*}
	and using the same argument with $\tau$ in place of $\sigma$ yields $$\nu_2(\phi_2^{k-1}(\sigma(r))) = \nu_2(\phi_2^{k-1}(\tau(r))).$$
	Again, \eqref{UniqueKids} in the definition of factorization tree gives $\phi_2^{k-1}(\sigma(r)) = \phi_2^{k-1}(\tau(r))$,
	a contradiction.
\end{proof}

\subsection{Primitive Factorization Trees}

We regard Theorem \ref{PrimitiveComplete} as the main result of this subsection and one of the primary results of this article.  As we shall see, its proof is quite
involved, particularly that of \eqref{UniqueSqF}.  Before proceeding, we must establish the preliminary result that
the maximal primitive factorization tree for $\alpha$ is unique up to isomorphism.

\begin{proof}[Proof of Theorem \ref{CPFTIso}]
	Suppose that $\ttt_1 = (T_1,\nu_1)$ and $\ttt_2 = (T_2,\nu_2)$.  Let
	\begin{equation*}
		V_n(\ttt_1) = \left\{r\in V(\ttt_1): \nu(r)\mbox{ is a factorization of } \alpha_n\right\}
	\end{equation*}
	and define $V_n(\ttt_2)$ in an analogous way.  So $V(\ttt_1)$ an $V(\ttt_2)$ may be written as the disjoint unions
	\begin{equation} \label{Disjoints}
		V(\ttt_1) = \bigcup_{n=0}^N V_n(\ttt_1)\quad\mathrm{and}\quad V(\ttt_2) = \bigcup_{n=0}^N V_n(\ttt_2).
	\end{equation}
	Let $\phi_1$ and $\phi_2$ be the parenting maps for $\ttt_1$ and $\ttt_2$, respectively.
	We shall recursively define maps $\sigma_n:V_n(\ttt_1) \to V_n(\ttt_2)$ satisfying
	\begin{enumerate}[(i)]
		\item\label{Roots} $\nu_1(r) = \nu_2(\sigma_n(r))$ for all $r\in V_n(\ttt)$.
		\item\label{NBijection} $\sigma_n$ is a surjection.
		\item\label{SharingParents} If $n\geq 1$ then $\phi_2(\sigma_n(r)) = \sigma_{n-1}(\phi_1(r))$ for all $r\in V_n(\ttt_1)$.
	\end{enumerate}
	
	By definition of factorization tree, $\ttt_1$ and $\ttt_2$ each have one root vertex $r_0$ and $s_0$, respectively, and $$\nu_1(r_0) = \nu_2(s_0) = (1,1,\ldots).$$  
	Now we may define
	\begin{equation*}
		\sigma_0(r_0) = s_0.
	\end{equation*}
	Condition \eqref{Roots} follows from the fact that $\nu_1(r_0) = \nu_2(s_0)$, while \eqref{NBijection} is trivial and \eqref{SharingParents} is vacuously correct.
	
	Suppose that $\sigma_{n-1}: V_{n-1}(\ttt_1) \to V_{n-1}(\ttt_2)$ satisfies properties \eqref{Roots}, \eqref{NBijection}
	and \eqref{SharingParents} with $n-1$ in place of $n$.  Assume that $r \in V_n(\ttt_1)$ .  By definition of maximal primitive factorization tree, 
	we know that $\nu_1(r) \in \Delta(\nu_1(\phi_1(r)))$.  Since $\phi_1(r) \in V_{n-1}(\ttt_1)$, property \eqref{Roots} gives 
	\begin{equation*}
		\nu_1(r) \in \Delta(\nu_2(\sigma_{n-1}(\phi_1(r)))).
	\end{equation*} 
	By definition of maximal primitive factorization tree, this yields
	\begin{equation*}
		\nu_1(r) \in \mathcal \nu_2(\mathcal C(\sigma_{n-1}(\phi_1(r)))).
	\end{equation*}
	Therefore, there must exist a child $s$ of $\sigma_{n-1}(\phi_1(r))$ such that $\nu_2(s) = \nu_1(r)$.  By condition \eqref{UniqueKids} in the definition of factorization tree,
	we know there is precisely one child $s$ of $\sigma_{n-1}(\phi_1(r))$ such that $\nu_2(s) = \nu_1(r)$.  Therefore, we may define
	\begin{equation*}
		\sigma_n(r) = s.
	\end{equation*}
	Condition \eqref{Roots} follows directly from this definition and also, since $s$ is a child of $\sigma_{n-1}(\phi_1(r))$, we get that
	\begin{equation*}
		\phi_2(\sigma_n(r)) = \phi_2(s) = \sigma_{n-1}(\phi_1(r))
	\end{equation*}
	verifying \eqref{SharingParents}, so it remains only to establish \eqref{NBijection}.
	
	Let $s\in V_n(\ttt_2)$ so that $\phi_2(s) \in V_{n-1}(\ttt_2)$.  We have assumed that $\sigma_{n-1}$ is surjective, so there exists $r'\in V_{n-1}(\ttt_1)$
	such that $\sigma_{n-1}(r') = \phi_2(s)$.  Also, since $s$ is a child of $\phi_2(s)$, we know that $\nu_2(s) \in \Delta(\nu_2(\phi_2(s)))$, and hence
	\begin{equation*}
		\nu_2(s) \in \Delta(\nu_2(\sigma_{n-1}(r'))) = \Delta(\nu_1(r')) = \nu_1(\mathcal C(r')).
	\end{equation*}
	Hence, there must exist a child $r$ of $r'$ such that $\nu_1(r) = \nu_2(s)$, so it follows that $\sigma_n(r) = s$, proving that $\sigma_n$ is surjective.
	
	
	For each $0\leq n\leq N$, we have exhibited the existence of a map $\sigma_n:V_n(\ttt_1) \to V_n(\ttt_2)$ which satisfies \eqref{Roots}, \eqref{NBijection}
	and \eqref{SharingParents} for all $n$.  Since \eqref{Disjoints} are disjoint unions, we may define $\sigma:V(\ttt_1) \to V(\ttt_2)$ by
	\begin{equation*}
		\sigma(r) = \begin{cases} \sigma_0(r) & \mathrm{if}\ r \in V_0(\ttt_1) \\
							\sigma_1(r) & \mathrm{if}\ r \in V_1(\ttt_1) \\
							\vdots & \vdots \\
							\sigma_N(r) & \mathrm{if}\ r \in V_N(\ttt_1). \end{cases}
	\end{equation*}
	It is now straightforward to verify that $\sigma$ is a surjective homomorphism, and it follows from Corollary \ref{UniqueHMCor} that $\sigma$ is an isomorphism.
\end{proof}

Before we are able to  prove Theorem \ref{PrimitiveComplete}, we must establish a series of results concerning the case where $\alpha$ is square-free.
These will be used as part of the proof of \eqref{UniqueSqF}.  The definition of 
subfactorization states that if $\aaa < \bbb$ then there exists a set of factorizations $\aaa_{n+1},\aaa_{n+2},\ldots,\aaa_{m-1}$ of $\alpha_{n+1},\alpha_{n+2},\ldots,\alpha_{m-1}$, 
respectively, such that $$\aaa < \aaa_{n+1} < \aaa_{n+2} < \cdots < \aaa_{m-1} < \bbb.$$  If $\alpha$ is square-free, then this set is unique.

\begin{lem} \label{SubUnique}
	Suppose $\alpha$ is a square-free rational number and $n < m$.  Assume that $\aaa$ and $\bbb$ are factorizations of $\alpha_n$ and $\alpha_m$, respectively, 
	satisfying $\aaa < \bbb$.  There exists a unique set of factorizations $\aaa_{n+1},\aaa_{n+2},\ldots,\aaa_{m-1}$ of $\alpha_{n+1},\alpha_{n+2},\ldots,\alpha_{m-1}$, respectively, 
	such that $\aaa < \aaa_{n+1} < \aaa_{n+2} < \cdots < \aaa_{m-1} < \bbb$.
\end{lem}
\begin{proof}
	Suppose $\alpha = a/b$ where $a$ and $b$ are relatively prime positive integers, both square-free.
	It is sufficient to show that every factorization $\bbb$ of $\alpha_m$ has a unique direct subfactorization.  To see this, let
	\begin{equation*} \label{StartingSub}
		\bbb = \left(\frac{c_1}{d_1},\frac{c_2}{d_2},\cdots\right).
	\end{equation*}
	Assume without loss of generality that $p_m$ divides the numerator of $\alpha$ and suppose that $k$ is the unique index such that $p_m\mid c_k$.
	Suppose that $\aaa$ and $\aaa'$ are direct subfactorizations of $\bbb$.  If we let
	\begin{equation*}
		\aaa = \left(\frac{a_1}{b_1},\frac{a_2}{b_2},\cdots\right)\mbox{ and } \aaa' = \left(\frac{a'_1}{b'_1},\frac{a'_2}{b'_2},\cdots\right)
	\end{equation*}
	then by definition of subfactorization, we must have $b_i = d_i = b'_i$ for all $i$ and $a_i = c_i = a'_i$ for all $i\ne k$.
	We also know that
	\begin{equation*}
		\prod_{i=1}^\infty \frac{a_i}{b_i} = \alpha_n = \prod_{i=1}^\infty \frac{a'_i}{b'_i} 
	\end{equation*}
	which now means that $a_k = a'_k$.  We then get that $\aaa = \aaa'$ and the result follows by induction.

\end{proof}

In spite of condition \eqref{UniqueKids} in the definition of factorization tree, it is possible for a factorization tree to contain the same factorization in two distinct vertices.
When $\alpha$ is square-free, this situation cannot occur.

 \begin{lem} \label{Injection}
	If $\alpha$ is a square-free rational number and $\ttt = (T,\nu)$ is a factorization tree for $\alpha$ then $\nu: V(T) \to \mathfrak F_\alpha$ is an injection.
\end{lem}
\begin{proof}[Proof of Theorem \ref{Injection}]
	Suppose that $r,s\in V(T)$ and $\nu(r) = \nu(s)$.
	Let $r_0 = s_0$ be the root vertex of $\mathcal T$, $r_n = r$, and $s_n = s$.  Further assume that $\{r_0,r_1,\ldots,r_n\}$ and $\{s_0,s_1,\ldots,s_n\}$ are finite sequences
	of vertices of $\mathcal T$ such that $r_i$ and $s_i$ are parents of $r_{i+1}$ and $s_{i+1}$, respectively, for all $0 \leq i < n$.  By \eqref{ParentRelations} in the definition
	of factorization tree, we must have that
	\begin{equation*}
		\nu(r_0) < \nu(r_1) < \cdots < \nu(r_n)\quad\mathrm{and}\quad \nu(s_0) < \nu(s_1) < \ldots < \nu(s_n).
	\end{equation*}
	By Lemma \ref{SubUnique}, we know that $\nu(r_i) = \nu(s_i)$ for all $0\leq i \leq n$.
	
	We now prove by induction on $i$ that $r_i = s_i$ for all $i$.  Our assumption is that $r_0 = s_0$, so the base case is clear.  Now assume that $t = s_i = r_i$ for some $i < n$.
	We know that $s_{i+1}$ and $r_{i+1}$ are both children of $t$ such that $\nu(s_{i+1}) = \nu(r_{i+1})$.  By condition \eqref{UniqueKids}, it follows that
	$s_{i+1} = r_{i+1}$.  In particular, we have shown that $r=s$ as required.
\end{proof}

Assume that $\alpha$ is square-free and
\begin{equation*}
	\aaa = \left(\frac{a_1}{b_1},\cdots,\frac{a_{k-1}}{b_{k-1}},\frac{a_kp_n}{b_k},\frac{a_{k+1}}{b_{k+1}},\cdots\right)
\end{equation*}
is a factorization of $\alpha_n$ for some $n$.  In this case, Lemma \ref{Injection} implies that the content map
$\nu$ for a factorization tree $\ttt$ is one-to-one, so if $\aaa$ belongs to the image of $\nu$, we obtain
\begin{equation} \label{ParentFind}
	(\nu\phi\nu^{-1})(\aaa) = \left(\frac{a_1}{b_1},\cdots,\frac{a_{k-1}}{b_{k-1}},\frac{a_k}{b_k},\frac{a_{k+1}}{b_{k+1}},\cdots\right).
\end{equation}
We have an analogous observation in the case where $p_n$ divides the denominator of $\alpha_n$.

In general, there exist non-isomorphic factorization trees $\ttt_1$ and $\ttt_2$ satisfying $\nu_1(V(\ttt_1)) = \nu_2(V(\ttt_2))$.  
However, when $\alpha$ is  square-free, this equality is enough to conclude that $\ttt_1\cong\ttt_2$.

\begin{thm} \label{IsomTrees}
	Suppose that $\ttt_1 = (T_1,\nu_1)$ and $\ttt_2 = (T_2,\nu_2)$ are factorization trees for the square-free rational number $\alpha$.  Then $\nu_1(V(\ttt_1)) = \nu_2(V(\ttt_2))$ 
	if and only if $\ttt_1 \cong \ttt_2$.  In this case, $\nu_2^{-1}\nu_1:V(\ttt_1)\to V(\ttt_2)$ defines an isomorphism.
\end{thm}
\begin{proof}
	If $\ttt_1 \cong \ttt_2$ then we obtain $\nu_1(V(\ttt_1)) = \nu_2(V(\ttt_2))$ directly from the definition of isomorphism.  Hence, we assume that 
	$\nu_1(V(\ttt_1)) = \nu_2(V(\ttt_2))$.  By Lemma \ref{Injection}, we know that both $\nu_1$ and $\nu_2$ are one-to-one, so we
	may define $\sigma(r) = \nu_2^{-1}(\nu_1(r))$.   We claim that $\sigma$ is an isomorphism.  
	
	Clearly $\nu_1(r) = \nu_2(\sigma(r))$ and $\sigma$ is a surjection, so by Corollary \ref{UniqueHMCor}, it remains only to verify condition \eqref{ParentPreserving}
	of Theorem \ref{TreeHM}.  
	By Lemma \ref{Injection}, we know that $\nu_2$ is one-to-one, so it is enough to show that
	\begin{equation} \label{IsomEquality}
		(\nu_2\phi_2\sigma)(r) = (\nu_2\sigma\phi_1)(r) \mbox{ for all } r\in V^*(\ttt_1).
	\end{equation}
	Assume that $\nu_1(r)$ is a factorization of $\alpha_n$, and therefore, without loss of generality, we may write
	\begin{equation} \label{NurLabel}
		\nu_1(r) = \left(\frac{a_1}{b_1},\cdots,\frac{a_{k-1}}{b_{k-1}},\frac{a_kp_n}{b_k},\frac{a_{k+1}}{b_{k+1}},\cdots\right).
	\end{equation}
	This gives
	\begin{equation*}
		(\nu_2\phi_2\sigma)(r) = (\nu_2\phi_2\nu_2^{-1})(\nu_1(r)) = (\nu_2\phi_2\nu_2^{-1})
			\left(\frac{a_1}{b_1},\cdots,\frac{a_{k-1}}{b_{k-1}},\frac{a_kp_n}{b_k},\frac{a_{k+1}}{b_{k+1}},\cdots\right),
	\end{equation*}
	and applying \eqref{ParentFind}, we get
	\begin{equation} \label{FirstHalf}
		(\nu_2\phi_2\sigma)(r) = \left(\frac{a_1}{b_1},\cdots,\frac{a_{k-1}}{b_{k-1}},\frac{a_k}{b_k},\frac{a_{k+1}}{b_{k+1}},\cdots\right).
	\end{equation}
	On the other hand, we obtain from \eqref{NurLabel} that
	\begin{equation*}
		(\nu_2\sigma\phi_1)(r) = (\nu_1\phi_1)(r) = (\nu_1\phi_1\nu_1^{-1}) \left(\frac{a_1}{b_1},\cdots,\frac{a_{k-1}}{b_{k-1}},\frac{a_kp_n}{b_k},\frac{a_{k+1}}{b_{k+1}},\cdots\right).
	\end{equation*}
	Again using \eqref{ParentFind}, we find that
	\begin{equation*}
		(\nu_2\sigma\phi_1)(r) =  \left(\frac{a_1}{b_1},\cdots,\frac{a_{k-1}}{b_{k-1}},\frac{a_k}{b_k},\frac{a_{k+1}}{b_{k+1}},\cdots\right).
	\end{equation*}
	Combining this equality with \eqref{FirstHalf}, we obtain \eqref{IsomEquality} as required.
\end{proof}

In addition to the above facts regarding square-free numbers, the proof of Theorem \ref{PrimitiveComplete} will require an important property of $\Delta$.

\begin{lem}\label{PrimitiveDelta}
	Suppose $\alpha$ is a rational number and that $\aaa$ and $\bbb$ are factorizations of $\alpha_n$ and $\alpha_{n-1}$, respectively, for some $0 < n \leq N$.
	If $\bbb$ is a direct subfactorization of $\aaa$ and $\aaa$ is primitive then $\bbb$ is primitive and $\aaa\in \Delta(\bbb)$.
\end{lem}
\begin{proof}
	Assume that 
	\begin{equation} \label{NewStart}
		\bbb = \left( \frac{a_1}{b_1},\frac{a_2}{b_2},\cdots,\frac{a_\ell}{b_\ell}\right)
	\end{equation}
	where $a_\ell/b_\ell \ne 1$.  Assume without loss of generality that $p_{n}$ divides the numerator of $\alpha$.  We certainly have that
	\begin{enumerate}[(i)]
		\item If $p\mid a_i$ for some $i$ then $p\geq p_{n}$.
		\item\label{BBound} If $p\mid b_i$ for some $i$ then $p > p_{n}$.
	\end{enumerate}
	As $\bbb$ is a direct subfactorization of $\aaa$, $\aaa$ must have the form
	\begin{equation*}
		\aaa = \left(\frac{a_1}{b_1},\cdots,\frac{a_{k-1}}{b_{k-1}},\frac{a_kp_{n}}{b_k},\frac{a_{k+1}}{b_{k+1}},\cdots,\frac{a_\ell}{b_\ell}\right)
	\end{equation*}
	for some $k$ or 
	\begin{equation*}
		\aaa = \left( \frac{a_1}{b_1},\frac{a_2}{b_2},\cdots,\frac{a_\ell}{b_\ell},\frac{p_{n}}{1}\right).
	\end{equation*}
	To prove both claims, it is now sufficient to show that $a_kp_{n} < b_k$.  If this inequality fails, then since $\aaa$ is primitive, we must have $a_k=1$.
	We cannot have $b_k=1$ because this would contradict our assumption that \eqref{NewStart} is a factorization with $a_\ell/b_\ell \ne 1$.
	Therefore, $b_k > 1$ and \eqref{BBound} implies that $b_k > p_{n} = a_kp_{n}$, a contradiction.
\end{proof}

Equipped with the above results, we are now prepared to continue with the proof of Theorem \ref{PrimitiveComplete}.

\begin{proof}[Proof of Theorem \ref{PrimitiveComplete}]
	The first assertion follows inductively from the fact that $(1,1,\ldots)$ is primitive and, if $\aaa$ is primitive then $\Delta(\aaa)$ contains only primitive factorizations.
	This means that $\nu(V(P)) \subseteq \mathfrak P_\alpha$.
	
	To establish \eqref{PrimsAllOver}, it remains to show that $\mathfrak P_\alpha\subseteq \nu(V(P))$, so assume that $\aaa$ is a primitive factorization of $\alpha_n$. 
	We shall again use induction on $n$ to prove that $\aaa\in \nu(V(P))$.
	If $n=0$, then $\aaa$ is a factorization of $\alpha_0$ so that $\aaa = (1,1,\ldots)$ and the base case follows directly from the definition of factorization tree.
			
	Now assume that for every primitive factorization $\bbb$ of $\alpha_{n-1}$, we have $\bbb\in \nu(V(P))$.  Since $\aaa$ is a factorization of $\alpha_n$, we may write,
	without loss of generality, 
	\begin{equation*}
		\aaa =  \left(\frac{a_1}{b_1},\cdots,\frac{a_{k-1}}{b_{k-1}},\frac{a_kp_{n}}{b_k},\frac{a_{k+1}}{b_{k+1}},\cdots\right).
	\end{equation*}
	We also make the following assumptions without loss of generality directly from the definition of factorization.
	\begin{enumerate}[(a)]
		\item\label{SmallNum} If $p$ is any prime dividing $a_i$ for some $i$, then $p \geq p_n$. 
		\item\label{SmallDen} If $p$ is any prime dividing $b_i$ for some $i$, then $p > p_n$.
		\item $p_n$ does not divide $a_i$ or $b_i$ for any $i> k$.
		\item\label{Decreasing2} $\max\{a_kp_n,b_k\} \geq \max\{a_{k+1},b_{k+1}\}$.
	\end{enumerate}
	Now we let
	\begin{equation*}
		\bbb = \left(\frac{a_1}{b_1},\cdots,\frac{a_{k-1}}{b_{k-1}},\frac{a_k}{b_k},\frac{a_{k+1}}{b_{k+1}},\cdots\right).
	\end{equation*}
	To see that $\bbb$ is a factorization of $\alpha_{n-1}$, we note that all axioms are trivial except \eqref{Decreasing}, for which we need only show that
	$\max\{a_k,b_k\} \geq \max\{a_{k+1},b_{k+1}\}$.  If $b_k > p_na_k$, then this inequality simply follows from \eqref{Decreasing2}.
	Otherwise $p_na_k > b_k$, and since $\aaa$ is primitive, $a_k = 1$.  Since $p_n$ is strictly smaller than any prime dividing $b_k$, we get $b_k = 1$ as well.
	Therefore, \eqref{Decreasing2} gives $p_n \geq  \max\{a_{k+1},b_{k+1}\}$.  However, $p_n$ cannot divide either $a_{k+1}$ or $b_{k+1}$.  Combining these observations,
	if $q$ is a prime dividing $a_{k+1}$ or $b_{k+1}$, we must have $q < p_n$.  This contradicts \eqref{SmallNum} and \eqref{SmallDen} and forces
	$a_{k+1} = b_{k+1} = 1$, verifying the desired inequality.
	
	Using Lemma \ref{PrimitiveDelta}, we deduce that $\bbb$ is primitive.
	Hence, the inductive hypothesis implies that $\bbb\in \nu(V(P))$ and there must exist a vertex $r$ of $\ppp_\alpha$ such that $\nu(r) = \bbb$.
	In addition, Lemma \ref{PrimitiveDelta} asserts that $\aaa\in \Delta(\bbb)$.  Hence $\aaa\in \Delta(\nu(r))$, and by definition of maximal primitive factorization tree, 
	$\aaa\in \nu(\mathcal C(r)) \subseteq \nu(V(P))$.
	
	To show \eqref{CPFTBiggest}, let $\phi_1$ and $\phi_2$ be the parenting maps for $\ttt$ and $\ppp_\alpha$, respectively.
	We follow the proof of Theorem \ref{CPFTIso} very closely and use the same notation.
	In this case, we shall recursively define maps $\sigma_n:V_n(\ttt) \to V_n(\ppp_\alpha)$ satisfying
	\begin{enumerate}[(i)]
		\item\label{Rooties} $\nu_1(r) = \nu_2(\sigma_n(r))$ for all $r\in V_n(\ttt)$.
		\item\label{SharingParents2} If $n\geq 1$ then $\phi_2(\sigma_n(r)) = \sigma_{n-1}(\phi_1(r))$ for all $r\in V_n(\ttt)$.
	\end{enumerate}
	Letting $r_0$ and $s_0$ be the root vertices of $\ttt$ and $\ppp_\alpha$, respectively, we define $\sigma_0(r_0) = s_0$ exactly as in the proof of Theorem \ref{CPFTIso}.
	Both required conditions hold for $\sigma_0$.
	
	Assume that $\sigma_{n-1}: V_{n-1}(\ttt) \to V_{n-1}(\ppp_\alpha)$ satisfies properties \eqref{Rooties}
	and \eqref{SharingParents2} with $n-1$ in place of $n$.  Let $r \in V_n(\ttt)$.  
	By Lemma \ref{PrimitiveDelta}, we know that $\nu_1(r) \in \Delta(\nu_1(\phi_1(r)))$.   Since $\phi_1(r) \in V_{n-1}(\ttt)$, property \eqref{Rooties} gives 
	\begin{equation*}
		\nu_1(r) \in \Delta(\nu_2(\sigma_{n-1}(\phi_1(r)))).
	\end{equation*} 
	By definition of $\ppp_\alpha$, we have $\Delta(\nu_2(\sigma_{n-1}(\phi_1(r)))) = \nu_2(\mathcal C(\sigma_{n-1}(\phi_1(r))))$, and therefore
	\begin{equation*}
		\nu_1(r) \in \mathcal \nu_2(\mathcal C(\sigma_{n-1}(\phi_1(r)))).
	\end{equation*}
	Hence, there must exist a child $s$ of $\sigma_{n-1}(\phi_1(r))$ such that $\nu_2(s) = \nu_1(r)$.  
	By \eqref{UniqueKids} in the definition of factorization tree, there is precisely one child $s$ of $\sigma_{n-1}(\phi_1(r))$ such that $\nu_2(s) = \nu_1(r)$, and we define
	\begin{equation*}
		\sigma_n(r) = s.
	\end{equation*}
	Conditions \eqref{Rooties} and \eqref{SharingParents2} now follow using proofs identical to those used in the proof of Theorem \ref{CPFTIso}.
	$\sigma$ is also constructed in the same way.  The remaining assertions follow from Theorem \ref{UniqueHMThm}.
	
	Now we establish \eqref{UniqueSqF}.  In view of Theorem \ref{IsomTrees}, it is enough to show that $\nu_1(V(\ttt)) = \nu_2(V(\ppp_\alpha))$.
	By \eqref{CPFTBiggest}, there exists a homomorphism $\sigma:V(\ttt) \to V(\ppp_\alpha)$ which means that
	\begin{equation*}
		\nu_1(V(\ttt)) = \nu_2(\sigma(V(\ttt))) \subseteq \nu_2(V(\ppp_\alpha)).
	\end{equation*}
	Now assume that there exists a factorization $\aaa$ of $\alpha_n$, for some $n$, such that
	\begin{equation} \label{NewA}
		\aaa \in \nu_2(V(\ppp_\alpha)) \setminus \nu_2(\sigma(V(\ttt))).
	\end{equation}
	Let $r\in V(\ppp_\alpha)$ be such that $\nu_2(r) = \aaa$ and suppose $s$ is a descendant of $r$ in $\ppp_\alpha$ such that $\nu_2(s)$ is a factorization of $\alpha$.
	Therefore, we may assume that
	\begin{equation*}
		\phi_2^k(s) = r
	\end{equation*}
	for some positive integer $k$ with $0\leq k\leq N$.  By our assumptions, we know there exists a vertex $t$ of $\ttt$ such that $\nu_1(t) = \nu_2(s)$.
	Since $\sigma$ is a homomorphism, we conclude that $\nu_2(\sigma(t)) = \nu_2(s)$.  By Theorem \ref{Injection}, we know that $\nu_2$ is an injection,
	and hence, $\sigma(t) = s$.  It now follows that
	\begin{equation*}
		\phi_2^k(\sigma(t)) = \phi_2^k(s) = r,
	\end{equation*}
	and by using Theorem \ref{TreeHM} \eqref{ParentPreserving}, we obtain that
	\begin{equation*}
		\sigma(\phi_1^k(t)) = r.
	\end{equation*}
	But $\phi_1^k(t)\in V(\ttt)$ which contradicts \eqref{NewA}.
\end{proof}

\subsection{Optimal Factorization Trees}

We first must establish that every optimal factorization is primitive.

\begin{proof}[Proof of Theorem \ref{EOFPrim}]
	Assume that $\aaa =  \left(a_1/b_1,a_2/b_2,\ldots\right)$ is an optimal factorization of $\alpha$ which fails to be primitive.  
	Hence, there exists $k$ such that $\max\{a_k,b_k\}$ is composite, and assume without loss of generality that $a_k > b_k$.  
	Then write $a_k = cd$, where $c$ and $d$ are integers greater than $1$, and observe that
	\begin{equation} \label{AnonFact}
		\alpha =  \frac{a_1}{b_1} \cdots \frac{a_{k-1}}{b_{k-1}}\cdot \frac{c}{b_k}\cdot\frac{d}{1}\cdot\frac{a_{k+1}}{b_{k+1}}\cdots.
	\end{equation}
	It is straightforward to check that
	\begin{equation*}
		\limsup_{t\to \infty}\left( m\left(\frac{c}{b_k}\right)^t + m\left(d\right)^t\right)^{1/t} < m\left( \frac{a_k}{b_k}\right),
	\end{equation*}
	which means that there exists $T > 0$ such that
	\begin{equation*}
		m\left(\frac{c}{b_k}\right)^t + m\left(d\right)^t < m\left( \frac{a_k}{b_k}\right)^t
	\end{equation*}
	for all $t\geq T$.  Along with \eqref{AnonFact}, this contradicts our assumption that $\aaa$ is optimal.
\end{proof}

The proof of Theorem \ref{COFTIso} is very similar to that of Theorem \ref{CPFTIso} so we do not include it here.
The proof of Theorem \ref{OptimalOptimal} is based on a result which describes ancestors of optimal factorizations in an arbitrary primitive factorization tree.  
This result will require us to establish a lemma. 

\begin{lem} \label{Subs}
	Suppose $\alpha$ is a rational number and $\nu$ is the content map for $\ppp_\alpha$.  Assume that $r$ and $s$ are vertices of $\ppp_\alpha$ with $s$ a descendant of $r$.  
	Further assume that
	\begin{equation*}
		\nu(r) = \left(\frac{a_1}{b_1},\frac{a_2}{b_2},\ldots,\frac{a_\ell}{b_\ell},1,1,\ldots\right)
	\end{equation*}
	where $a_\ell/b_\ell \ne 1$ and 
	\begin{equation*}
		\nu(s) = \left(\frac{c_1}{d_1},\frac{c_2}{d_2},\ldots\right).
	\end{equation*}
	Then
	\begin{equation*}
		m\left( \frac{a_i}{b_i}\right) = m\left( \frac{c_i}{d_i}\right)
	\end{equation*}
	for all $1\leq i\leq \ell$.
\end{lem}
\begin{proof}
	We first establish the lemma under the additional assumption that $s$ is a child of $r$.  We may assume that $\nu(r)$ is a factorization of $\alpha_n$ so that
	$\nu(s)$ is a factorization of $\alpha_{n+1}$.  As usual, we assume without loss of generality that $p_{n+1}$ divides the numerator of $\alpha$.
	The definition of $\ppp_\alpha$ gives $\nu(s)\in \Delta(\nu(r))$.
	Therefore, we either have
	\begin{equation*}
		\nu(s) = \left(\frac{a_1}{b_1},\dots,\frac{a_{k-1}}{b_{k-1}},\frac{a_kp_{n+1}}{b_k},\frac{a_{k+1}}{b_{k+1}},\cdots,\frac{a_\ell}{b_\ell},1,1,\ldots\right)
			\mbox{  for some } 1\leq k\leq \ell
	\end{equation*}
	or
	\begin{equation*}
		\nu(s) =  \left( \frac{a_1}{b_1},\frac{a_2}{b_2},\ldots, \frac{a_\ell}{b_\ell},\frac{p_{n+1}}{1},1,1,\cdots\right)
	\end{equation*}
	In the second case, clearly the desired result holds.  In the first case, then $a_kp_{n+1} < b_k$ by the definition of $\delta(\nu(s))$, so the result holds as well.

	Suppose that 
	\begin{equation*}
		r = s_0,  s_1, s_2, \ldots s_{k-1}, s_k = s
	\end{equation*}
	are vertices of $\ttt$ such that $s_i$ is a parent of $s_{i+1}$ for all $i$.  The result now follows by induction on $i$.
	
\end{proof}

Suppose $\alpha$ is an arbitrary rational number and $n$ a positive integer with $0\leq n\leq N$.  We say that $n$ is a {\it separation index for $\alpha$} if
\begin{equation*}
	1 \in \left\{ \left|\alpha_n\right|_p, \left|\alpha/\alpha_n\right|_p\right\}\mbox{ for all primes } p.
\end{equation*}
It follows directly from this definition that $n$ is a separation index for $\alpha$ if and only if $p_n > p$ for every prime $p$ dividing the numerator or denominator
of $\alpha/\alpha_n$.  If $\ttt$ is a factorization tree for $\alpha$ and $r$ is a vertex of $\ttt$, we say that $r$ is a {\it separation vertex of $\ttt$} if
$\nu(r)$ is a factorization of $\alpha_n$ for some separation index $n$.  

\begin{thm} \label{OptimalAncestorsSep}
	Suppose $\alpha$ is a  rational number and $\ttt = (T,\nu)$ is a primitive factorization tree for $\alpha$ with parenting map $\phi$.	
	Let $r$ be an optimal vertex of $\ttt$.  If $s$ is a separation vertex of $\ttt$ and an ancestor of $r$ then $s$ is also optimal.
\end{thm}
\begin{proof}
	We assume $\nu(r)$ is a factorization of $\alpha_n$ and that $k$ is a positive integer such that $\phi^k(r) = s$.
	Certainly $s$ is a factorization of $\alpha_{n-k}$, so by our assumptions, $n-k$ is a separation index for $\alpha$.
	
	Assume that $\nu(\phi^k(r)) \not\in \mathfrak O_\alpha$ for some $k\leq n$.
	By Theorem \ref{PrimitiveComplete}, there exists an injective edge-preserving homomorphism $\sigma:V(T) \to V(\ppp_\alpha)$.
	Let $\ppp_\alpha = (P,\nu_2)$, so by Theorem \ref{TreeHM}, we know that 
	\begin{equation} \label{Outsider}
		\nu_2(\phi^k(\sigma(r))) \not\in \mathfrak O_\alpha.
	\end{equation}
	Write
	\begin{equation*}
		\nu_2(\phi^k(\sigma(r))) = \left(\frac{a_1}{b_1},\frac{a_2}{b_2},\cdots,\frac{a_\ell}{b_\ell},1,1,\cdots \right)
	\end{equation*}
	with $a_\ell/b_\ell \ne 1$ and
	\begin{equation*}
		\nu_2(\sigma(r)) = \left(\frac{a'_1}{b'_1},\frac{a'_2}{b'_2},\cdots \right).
	\end{equation*}
	It follows from Lemma \ref{Subs} that
	\begin{equation} \label{FirstMeasure}
		m\left(\frac{a_i}{b_i}\right) = m\left( \frac{a'_i}{b'_i}\right) \mbox{ for all } 1\leq i\leq \ell.
	\end{equation}
	
	Applying Corollary \ref{OptimalPrimitive}, we get that $\ppp_\alpha$ is optimal, and therefore, there must exist $t\in V(P)$ such that $\nu_2(t)$ is an optimal factorization
	of $\alpha_{n-k}$.  Now we may write
	\begin{equation*}
		\nu_2(t) =  \left(\frac{c_1}{d_1},\frac{c_2}{d_2},\ldots\right).
	\end{equation*}
	We cannot have $m(a_i/b_i) =m( c_i/d_i)$ for all $1\leq i\leq \ell$ because this would contradict \eqref{Outsider}.
	Now suppose that $j$ is the smallest index satisfying $m(a_j/b_j) \ne m(c_j/d_j)$.   Therefore,
	\begin{equation} \label{SecondMeasure}
		m\left( \frac{a_i}{b_i}\right) = m\left( \frac{c_i}{d_i}\right)\mbox{ for all } i < j.
	\end{equation}
	Again since $\nu_2(t)$ is optimal while $\nu_2(\phi^k(\sigma(r)))$ is not, we must have that $m(c_j/d_j) < m(a_j/b_j)$.  In particular, this forces $j \leq \ell$.
	
	Now assume that $t'$ is a descendant of $t$ in $\ppp_\alpha$ such that $\nu_2(t')$ is a factorization of $\alpha_n$ and write
	\begin{equation*}
		\nu_2(t') = \left(\frac{c'_1}{d'_1},\frac{c'_2}{d'_2},\ldots\right).
	\end{equation*}
	We know that $c_{i}/d_{i} \ne 1$ for all $i < j$.  Again using Lemma \ref{Subs}, we obtain that
	\begin{equation*}
		m\left( \frac{c_i}{d_i}\right) = m\left( \frac{c'_i}{d'_i}\right)\mbox{ for all } 1\leq i < j.
	\end{equation*}
	Combining this with \eqref{FirstMeasure} and \eqref{SecondMeasure}, we find that
	\begin{equation*}
		m\left( \frac{a'_i}{b'_i}\right) = m\left( \frac{c'_i}{d'_i}\right)\mbox{ for all } 1\leq i < j.
	\end{equation*}
	We now claim that $m(c'_j/d'_j) < m(a'_j/b'_j)$.  To see this, we first observe that
	\begin{equation*}
		m\left( \frac{a'_j}{b'_j}\right) = m\left(\frac{a_j}{b_j}\right) > m\left(\frac{c_j}{d_j}\right).
	\end{equation*}
	If $m(c_j/d_j) \ne 0$ then it follows from Lemma \ref{Subs} that $m(c_j/d_j) = m(c'_j/d'_j)$ and our claim follows.  Otherwise,
	$m(c_i/d_i) = 0$ for all $i\geq j$ and $m(c'_j/d'_j)$ must either equal $0$ or equal $p_i$ for some $i > n-k$.  But since $n-k$ is a separation index for $\alpha$, 
	we know that $p_{n-k} > p_i$, and we conclude that $m(a_j/b_j) >  m(c'_j/d'_j)$ establishing our claim.  The result now follows.
\end{proof}

We shall now recognize Theorem \ref{OptimalOptimal} as a corollary to the following result.

\begin{thm} \label{OptimalSeparation}
	Suppose $\alpha$ is a rational number and $n$ is a separation index for $\alpha$.  If $\aaa$ is an optimal factorization of $\alpha_n$ then $\aaa\in \nu(V(\ooo_\alpha))$.
\end{thm}
\begin{proof}
	Assume that
	\begin{equation*}
		0 = n_0 < n_1 < \cdots < n_{L-1} < n_L = N
	\end{equation*}
	is the complete list of separation indices for $\alpha$ so we know that $\aaa$ is a factorization of $\alpha_{n_\ell}$ for some $0\leq \ell \leq L$.
	We shall establish the theorem using induction on $\ell$.  If $\aaa$ is a factorization of $\alpha_0$ then $\aaa = (1,1,\ldots)$ and we know that $\aaa\in\nu(V(\ooo_\alpha))$.
	
	Now we assume that, for every optimal factorization $\bbb$ of $\alpha_{n_{\ell-1}}$, we have that $\bbb\in \nu(V(\ooo_\alpha))$.
	By our assumptions, we know that $\aaa$ is a optimal factorization of $\alpha_{n_\ell}$ for some $0\leq \ell \leq L$.  Therefore, Theorem \ref{EOFPrim} asserts that
	$A$ is primitive, so letting $\ppp_\alpha = (P,\nu')$, it follows from Theorem \ref{PrimitiveComplete} that there exists $r'\in V(P)$ such that $\nu'(r') = \aaa$.
	We may let $s'$ be an ancestor of $r'$ in $\ppp_\alpha$ such that $\nu'(s')$ is a factorization of $\alpha_{n_{\ell-1}}$.  Now set $\bbb = \nu'(s')$.
	
	Certainly we have that $\bbb < \aaa$, and since $\aaa$ is assumed to be optimal, Theorem \ref{OptimalAncestorsSep} implies that $\bbb$ is also optimal.
	Therefore, the inductive hypothesis yields a vertex $r\in V(\ooo_\alpha)$ such that $\nu(r) = \bbb$.  Let
	\begin{equation*}
		\bbb = \left(\frac{a_1}{b_1},\frac{a_2}{b_2},\cdots,\frac{a_k}{b_k}\right),\ \mathrm{where}\ \frac{a_k}{b_k}\ne 1.
	\end{equation*}
	Since $n_{\ell - 1}$ and $n_\ell$ are consecutive separation indices, there must exist exactly one prime $p$ dividing the numerator or denominator of $\alpha_{n_\ell}$ that
	divides neither the numerator nor denominator of $\alpha_{n_{\ell-1}}$.  This prime cannot divide both the numerator and denominator of $\alpha_{n_\ell}$, 
	so we shall assume without loss of generality that $p$ divides the numerator of $\alpha_{n_\ell}$.
	We now know there exist non-negative integers $j_1,j_2,\ldots,j_k$ and $z$ such that
	\begin{equation*}
		\aaa = \left(\frac{a_1p^{j_1}}{b_1},\frac{a_2p^{j_2}}{b_2},\cdots,\frac{a_kp^{j_k}}{b_k}, \underbrace{\frac{p}{1},\cdots, \frac{p}{1}}_{z\ \mathrm{times}}\right).
	\end{equation*}
	If $z = 0$ then clearly $\aaa\in\nu(V(\ooo_\alpha))$ as required.  If $z > 0$, we must have that
	\begin{equation} \label{AComp}
		a_ip^{j_i+1} > b_i\mbox{ for all } 1\leq i\leq k.
	\end{equation}
	since otherwise, $a_ip^{j_i+1} < b_i$ for some $i$ and then
	\begin{equation*}
		\left(\frac{a_1p^{j_1}}{b_1},\cdots,\frac{a_ip^{j_i+1}}{b_i},\cdots ,\frac{a_kp^{j_k}}{b_k}, \underbrace{\frac{p}{1},\cdots, \frac{p}{1}}_{z-1\ \mathrm{times}}\right)
	\end{equation*}
	is another primitive factorization of $\alpha_{n_\ell}$.  This would mean that $\aaa$ is not optimal, contradicting our assumption.  Letting
	$\gamma = j_1 + \cdots + j_k$, we certainly have that 
	\begin{equation*}
		\left(\frac{a_1p^{j_1}}{b_1},\frac{a_2p^{j_2}}{b_2},\cdots,\frac{a_kp^{j_k}}{b_k}\right)
	\end{equation*}
	is a factorization of $\alpha_{\gamma+n_{\ell-1}}$ and clearly belongs to $\nu(V(\ooo_\alpha))$.  It follows from \eqref{AComp} that $\aaa\in\nu(V(\ooo_\alpha))$ 
	as required.

\end{proof}

Theorem \ref{OptimalOptimal} now follows from the fact that $n$ is always a separation index for $\alpha_n$ and the fact that first $n+1$ generations of $\ooo_\alpha$
exactly form the tree data structure $\ooo_{\alpha_n}$.

\subsection{Measure Class Graphs} 

Finally, we establish our two main results regarding measure class graphs beginning with Theorem \ref{MCGUnique}.

\begin{proof}[Proof of Theorem \ref{MCGUnique}]
	Suppose that $\ttt_1 = (T_1,\nu_1)$, $\ttt_2 = (T_2,\nu_2)$, $\G_1 = (G_1,\mu_1)$ and $\G_2 = (G_2,\mu_2)$.  Also suppose that $\pi_1$ and $\pi_2$
	are the respective projection maps.  Assume that $g\in V(G_1)$ and that $r,s\in \pi_1^{-1}(g)$.  We observe that
	\begin{equation*}
		\mu_2(\pi_2(\sigma(r))) = f(\nu_2(\sigma(r))) = f(\nu_1(r)) = \mu_1(\pi_1(r)) = \mu_1(\pi_2(s)).
	\end{equation*}
	Then using the analogous equalities for $s$, we conclude that $\mu_2(\pi_2(\sigma(r))) = \mu_2(\pi_2(\sigma(s)))$.  Since $\mu_2$ is injective,
	we obtain that $\pi_2(\sigma(r)) = \pi_2(\sigma(s))$.  We may now define $\tau:V(G_1) \to V(G_2)$ by $\tau = \pi_2\sigma\pi_1^{-1}$.
	By our above work, $\tau$ is well-defined.
	
	It remains only to show that $\tau$ is bijective and edge-preserving.  If we take $g\in V(G_2)$ and let $r\in \pi_2^{-1}(g)$ then
	\begin{equation*}
		\tau((\pi_1\sigma^{-1})(r)) = g
	\end{equation*}
	meaning that $\tau$ is surjective.  Furthermore, $\pi_1\sigma^{-1}\pi_2^{-1}$ is a well-defined map from $V(G_2)$ to $V(G_1)$ by the same argument
	as above, which means that $\tau$ is injective.
	
	Finally, assume that $(g,h)\in E(G_1)$.  Since $\pi_1$ is surjective, we know that $g,h\in \pi_1(V(T_1))$.  Therefore, since $\pi_1$ is faithful, there exists
	$(r,s)\in E(T_1)$ such that $(\pi_1(r),\pi_1(s)) = (g,h)$.  We certainly have that $\pi_2\sigma$ satisfies property \eqref{EdgePreserving} in the definition
	of homomorphism, so it follows that
	\begin{equation*}
		(\pi_2(\sigma(r)),\pi_2(\sigma(s))) \in E(G_2).
	\end{equation*}
	We then obtain $(\tau(r),\tau(s))\in E(G_2)$.  If we assume that $(\tau(r),\tau(s))\in E(G_2)$ then $(g,h)\in E(G_1)$ follows by a similar argument.
\end{proof}

Finally, we prove that if $\alpha$ is square-free and $\ttt$ is a primitive factorization tree for $\alpha$ then the measure class graph for $\ttt$
is a binary tree.

\begin{proof}[Proof of Theorem \ref{BinTree}]
	Assume that $\ttt = (T,\nu)$ and $r,s\in V(T)$ are such that $\nu(r) \sim \nu(s)$.  We must first show that $\nu(\phi(r)) \sim \nu(\phi(s))$.
	
	To see this, since $\nu(r)$ and $\nu(s)$ are measure equivalent, we may assume that
	\begin{equation*}
		\nu(r) = \left(\frac{a_1}{b_1},\frac{a_2}{b_2},\cdots,\frac{a_\ell}{b_\ell}\right)\mbox{ and } \nu(s) = \left(\frac{c_1}{d_1},\frac{c_2}{d_2},\cdots,\frac{c_\ell}{d_\ell}\right)
	\end{equation*}
	are factorizations of $\alpha_n$ with $a_i/b_i \ne 1$ and $c_i/d_i\ne 1$ for all $i$.  In addition, we know that
	\begin{equation} \label{EqualMax}
		\max\{a_i,b_i\} = \max\{c_i,d_i\}\quad\mbox{for all } i
	\end{equation}
	Assume without loss of generality that $p_n$ divides the numerator of $\alpha = a/b$.  We know that $p_n$ must divide $a_i$ for some $i$.
	Now consider two cases.
	
	First assume that $p_n = \max\{a_i,b_i\}$.  Since $\nu(r)\sim \nu(s)$ we must have that $p_n = c_i$, and
	the definition of factorization implies that that $i=\ell$.  Since $p_n$ is larger than any prime dividing $b$, we have $b_\ell = d_\ell = 1$.  Hence,
	since $\alpha$ is square-free, we must have
	\begin{equation*}
		\nu(\phi(r)) = \left(\frac{a_1}{b_1},\frac{a_2}{b_2},\cdots,\frac{a_{\ell-1}}{b_{\ell-1}}\right)\mbox{ and } 
			\nu(\phi(s)) = \left(\frac{c_1}{d_1},\frac{c_2}{d_2},\cdots,\frac{c_{\ell-1}}{d_{\ell-1}}\right),
	\end{equation*}
	and it follows from \eqref{EqualMax} that these factorizations are measure equivalent.
	
	Now assume that $p_n\ne\max\{a_i,b_i\}$.  Since $\nu(r)$ is primitive, we must have $a_i < b_i$ so that $\max\{a_i/p,b_i\} = \max\{a_i,b_i\}$.
	We may also assume that $p_n$ divides $c_j$, and since $\nu(r)\sim\nu(s)$, we conclude that $c_j < d_j$ and $\max\{c_j/p,d_j\} = \max\{c_j,d_j\}$.
	Using \eqref{EqualMax} and the fact that $\alpha$ is square-free we obtain that $\nu(\phi(r)) \sim \nu(\phi(s))$ establishing our claim.

	It follows now that $\overline \ttt$ is a tree, but we must still show that it is binary.  To see this, assume that $\aaa$ is a primitive factorization
	of $\alpha_n$ given by
	\begin{equation*}
		\aaa = \left( \frac{a_1}{b_1},\frac{a_2}{b_2},\cdots,\frac{a_\ell}{b_\ell}\right)
	\end{equation*}
	where $a_\ell/b_\ell \ne 1$.  If $\bbb$ is such that $\aaa$ is a direct subfactorization of $\bbb$ then
	\begin{equation*}
		m(\bbb) \in \left\{ m(\aaa), \left( m\left(\frac{a_1}{b_1}\right),m\left(\frac{a_2}{b_2}\right),\cdots,m\left(\frac{a_\ell}{b_\ell}\right),\log p_{n+1} \right)\right\}.
	\end{equation*}
	In particular, $m(\bbb)$ depends only on $m(\aaa)$ and $n$.  We now conclude that $\overline \ttt$ is binary as required.
\end{proof}

\section{Examples} \label{Examples}

Our original goal for developing factorization trees was to determine optimal factorizations of rational numbers.  Using the techniques presented in \cite{SamuelsMetric},
we can search the leaf vertices of $\ooo_\alpha$ to find all optimal factorizations of $\alpha$.
In this section, we provide several examples of this strategy.

\begin{ex} \label{30/7Again}
With the assistance of \cite{JankSamuels}, the work of \cite{SamuelsParametrized} establishes that $(5/7,3,2)$ is an optimal factorization of $\alpha = 30/7$.
By using the canonical optimal factorization tree for $30/7$, we obtain a new proof.  Indeed, $\ooo_{30/7}$ is as follows.

\bigskip
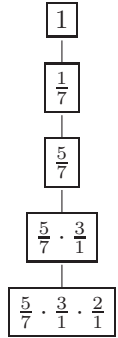
\Tree[.\fbox{$1$} 
		[.\fbox{$\frac{1}{7}$} 
			[.\fbox{$\frac{5}{7}$} 
				[.\fbox{$\frac{5}{7}\cdot \frac{3}{1}$}
					\fbox{$\frac{5}{7}\cdot\frac{3}{1}\cdot \frac{2}{1}$}
				]
			]
		]
	]
\bigskip

As there is only one leaf node of $\ooo_{30/7}$, and we know that $\ooo_{30/7}$ is optimal, the factorization $\frac{30}{7} = \frac{5}{7}\cdot\frac{3}{1}\cdot \frac{2}{1}$
must be the only optimal factorization of $30/7$.
\end{ex}

\begin{ex}
Now consider
\begin{equation*}
	\alpha = \frac{851}{858} = \frac{37\cdot 23}{13\cdot 11\cdot 3\cdot 2}
\end{equation*}
Once again, $\alpha$ is square-free so we need only use $\ooo_\alpha$ to locate all optimal factorization of $\alpha$.

\bigskip
\Tree[.\fbox{$1$} 
		[.\fbox{$\frac{37}{1}$} 
			[.\fbox{$\frac{37}{1}\cdot\frac{23}{1}$} 
				[.\fbox{$\frac{37}{13}\cdot \frac{23}{1}$}
					[.\fbox{$\frac{37}{13}\cdot\frac{23}{11}$}
						[.\fbox{$\frac{37}{13}\cdot\frac{23}{11}\cdot \frac{1}{3}$}
							\fbox{$\frac{37}{26}\cdot\frac{23}{11}\cdot \frac{1}{3}$}
							\fbox{$\frac{37}{13}\cdot\frac{23}{22}\cdot \frac{1}{3}$}
						]
					]
				]
				[.\fbox{$\frac{37}{1}\cdot \frac{23}{13}$}
					[.\fbox{$\frac{37}{11}\cdot \frac{23}{13}$}
						[.\fbox{$\frac{37}{33}\cdot \frac{23}{13}$}
							\fbox{$\frac{37}{33}\cdot \frac{23}{13}\cdot \frac{1}{2}$}
						]
					]
				]
			]
		]
	]
\bigskip

Among the leaf nodes, all three factorizations have equal Mahler measures in the first two entries, while the factorization
\begin{equation} \label{851/858}
	\frac{851}{858} = \frac{37}{33}\cdot \frac{23}{13}\cdot \frac{1}{2}
\end{equation}
has the smallest Mahler measure in the third entry.  No other factorization appearing among leaf nodes has measure equal to that of
\eqref{851/858}, so it follows that \eqref{851/858} is the only optimal factorization of $851/858$.  If we were interested only in the measure of an optimal
factorization of $\alpha$, we could instead look the measure class graph for $\ooo_\alpha$ given as follows.

\bigskip
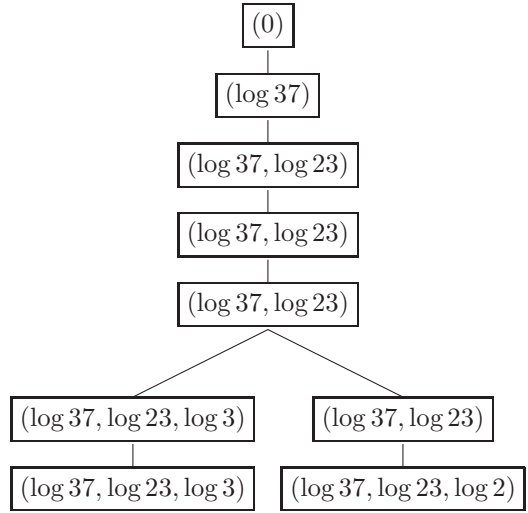
\Tree[.\fbox{$(0)$}
		[.\fbox{$(\log 37)$}
			[.\fbox{$(\log 37,\log 23)$}
				[.\fbox{$(\log 37,\log 23)$}
					[.\fbox{$(\log 37,\log 23)$}
						[.\fbox{$(\log 37,\log 23,\log 3)$}
							\fbox{$(\log 37,\log 23,\log 3)$}
						]
						[.\fbox{$(\log 37,\log 23)$}
							\fbox{$(\log 37,\log 23,\log 2)$}
						]
					]
				]
			]
		]
	]
\bigskip

Now we easily see that any optimal factorization of $\alpha$ has measure given by $(\log 37,\log 23,\log 2)$.

\end{ex}

\begin{ex}\label{Complicated} As our final example, we consider
\begin{equation*}
	\alpha = \frac{316,889}{549,010} = \frac{131\cdot 59\cdot 41}{31\cdot 23\cdot 11\cdot 7\cdot 5\cdot 2}
\end{equation*}
The first eight generations of $\ooo_\alpha$ are given as follows.

\bigskip
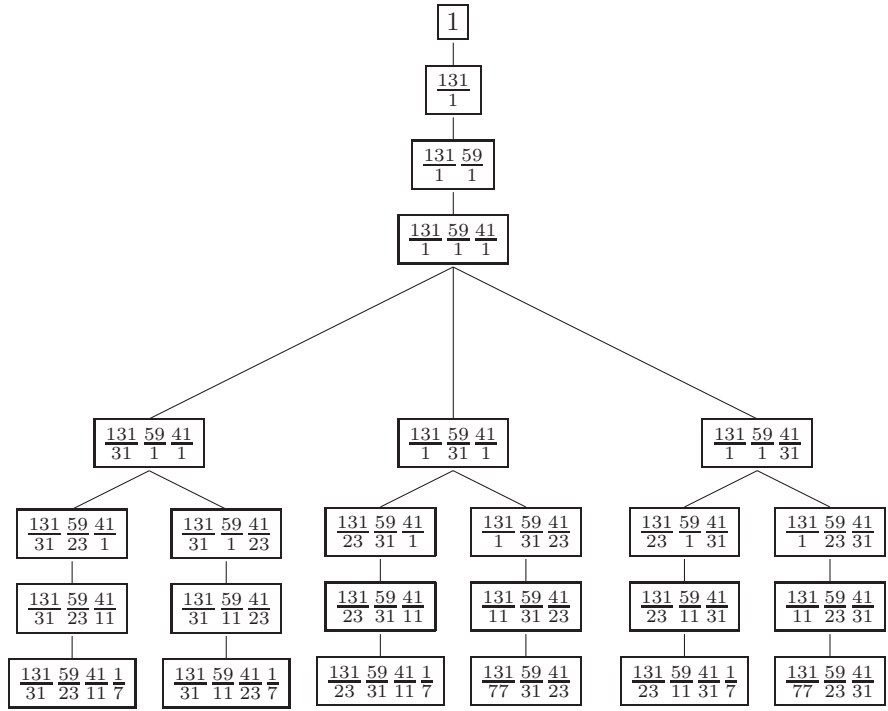
\Tree[.\fbox{$1$} 
		[.\fbox{$\frac{131}{1}$} 
			[.\fbox{$\frac{131}{1}\frac{59}{1}$} 
				[.\fbox{$\frac{131}{1}\frac{59}{1} \frac{41}{1}$} 
					[.\fbox{$\frac{131}{31}\frac{59}{1}\frac{41}{1}$} 
						[.\fbox{$\frac{131}{31}\frac{59}{23}\frac{41}{1}$} 
							[.\fbox{$\frac{131}{31} \frac{59}{23}  \frac{41}{11}$} 
								[.\fbox{$\frac{131}{31} \frac{59}{23}  \frac{41}{11}  \frac{1}{7}$} 
								]
							]
						]
						[.\fbox{$\frac{131}{31} \frac{59}{1}  \frac{41}{23}$} 
							[.\fbox{$\frac{131}{31} \frac{59}{11}  \frac{41}{23}$} 
								[.\fbox{$\frac{131}{31} \frac{59}{11}  \frac{41}{23}  \frac{1}{7}$} 
								]
							]
						]
					]
					[.\fbox{$\frac{131}{1} \frac{59}{31}  \frac{41}{1}$} 
						[.\fbox{$\frac{131}{23} \frac{59}{31}  \frac{41}{1}$} 
							[.\fbox{$\frac{131}{23} \frac{59}{31}  \frac{41}{11}$} 
								[.\fbox{$\frac{131}{23} \frac{59}{31}  \frac{41}{11}  \frac{1}{7}$} 
								]
							]
						]
						[.\fbox{$\frac{131}{1} \frac{59}{31}  \frac{41}{23}$} 
							[.\fbox{$\frac{131}{11} \frac{59}{31}  \frac{41}{23}$} 
								[.\fbox{$\frac{131}{77} \frac{59}{31}  \frac{41}{23}$} 
								]
							]
						]
					]
					[.\fbox{$\frac{131}{1} \frac{59}{1}  \frac{41}{31}$} 
						[.\fbox{$\frac{131}{23} \frac{59}{1}  \frac{41}{31}$} 
							[.\fbox{$\frac{131}{23} \frac{59}{11}  \frac{41}{31}$} 
								[.\fbox{$\frac{131}{23} \frac{59}{11}  \frac{41}{31}  \frac{1}{7}$} 
								]
							]
						]
						[.\fbox{$\frac{131}{1} \frac{59}{23}  \frac{41}{31}$} 
							[.\fbox{$\frac{131}{11} \frac{59}{23}  \frac{41}{31}$} 
								[.\fbox{$\frac{131}{77} \frac{59}{23}  \frac{41}{31}$} 
								]
							]
						]
					]
				]
			]
		]
	]
	
\bigskip
	
In view of Theorem \ref{OptimalAncestorsSep}, any optimal factorization of $\alpha$ must be a descendant of the vertex containing either $(\frac{131}{77}, \frac{59}{31},  \frac{41}{23})$
or $(\frac{131}{77}, \frac{59}{23},  \frac{41}{31})$.  Hence, we consider only the portion of $\ooo_\alpha$ having descendants of these vertices.

\bigskip
\begin{enumerate}
\item[] 
\Tree[.\fbox{$\frac{131}{77} \frac{59}{31} \frac{41}{23}$} 
		[.\fbox{$\frac{131}{77} \frac{59}{31}  \frac{41}{23}\frac{1}{5}$} 
			[.\fbox{$\frac{131}{77} \frac{59}{31}  \frac{41}{23} \frac{1}{5}\frac{1}{2}$} 
			]
		]
	]
\Tree[.\fbox{$\frac{131}{77} \frac{59}{23}  \frac{41}{31}$} 
		[.\fbox{$\frac{131}{77} \frac{59}{23}  \frac{41}{31}\frac{1}{5}$} 
			[.\fbox{$\frac{131}{77} \frac{59}{46}  \frac{41}{31} \frac{1}{5}$} 
			]
		]
	]
\end{enumerate}
\bigskip
Therefore, we conclude that 
\begin{equation*}
	\alpha =  \frac{316,889}{549,010} = \frac{131}{77}\cdot \frac{59}{46} \cdot \frac{41}{31}\cdot \frac{1}{5} 
\end{equation*}
is the only optimal factorization of $\alpha$.  

\end{ex}

\end{document}